\documentclass[a4paper, 11pt]{article}
\usepackage{amsmath,amssymb,esint,amscd,xspace,fancyhdr,color,authblk,srcltx,fontenc,bbm}
\usepackage{hyperref}
\usepackage{cite}
\usepackage[title]{appendix}

\newtheorem{theorem}{Theorem}[section]

\newtheorem{definition}[theorem]{Definition}
\newtheorem{lemma}[theorem]{Lemma}
\newtheorem{proposition}[theorem]{Proposition}
\newtheorem{remark}[theorem]{Remark}

\newtheorem{taggedtheoremx}{Theorem}
\newenvironment{taggedtheorem}[1]
 {\renewcommand\thetaggedtheoremx{#1}\taggedtheoremx}
 {\endtaggedtheoremx}
 
\newenvironment{proof}[1][Proof]{\textbf{#1.} }{\hfill\rule{0.5em}{0.5em}}
{\catcode`\@=11\global\let\AddToReset=\@addtoreset
\AddToReset{equation}{section}

\AddToReset{theorem}{section}

\title{Global gradient estimates for very singular quasilinear elliptic equations with measure data}

\author{Minh-Phuong Tran\thanks{Applied Analysis Research Group, Faculty of Mathematics and Statistics, Ton Duc Thang University, Ho Chi Minh city, Vietnam; \texttt{tranminhphuong@tdtu.edu.vn}}, Thanh-Nhan Nguyen\thanks{Group of Analysis and Applied Mathematics, Department of Mathematics, Ho Chi Minh City University of Education, Ho Chi Minh city, Vietnam; \texttt{nhannt@hcmue.edu.vn}}}

\date{\today} 

\begin{document}
\maketitle

\begin{abstract}
This paper continues the development of regularity results for quasilinear measure data problems
\begin{align*}
\begin{cases}
-\mathrm{div}(A(x,\nabla u)) &= \mu \quad \text{in} \ \ \Omega, \\
\quad \quad \qquad u &=0 \quad \text{on} \ \ \partial \Omega,
\end{cases}
\end{align*}
in Lorentz and Lorentz-Morrey spaces, where $\Omega \subset \mathbb{R}^n$ ($n \ge 2$), $\mu$ is a finite Radon measure on $\Omega$, and $A$ is a monotone Carath\'eodory vector valued operator acting between $W^{1,p}_0(\Omega)$ and its dual $W^{-1,p'}(\Omega)$. It emphasizes that this paper studies the `very singular' case $1<p \le \frac{3n-2}{2n-1}$ and the problem is considered under the weak assumption, where the $p$-capacity uniform thickness condition is imposed on the complement of domain $\Omega$. There are two main results obtained in our study pertaining to the global gradient estimates of solutions in Lorentz and Lorentz-Morrey spaces involving the use of  maximal and fractional maximal operators. The idea for writing this working paper comes directly from the recent results by others in the same research topic, where global estimates for gradient of solutions for the `very singular' case still remains a challenge, specifically related to Lorentz and Lorentz-Morrey spaces.

\medskip

\medskip

\end{abstract} 

\newpage
\tableofcontents  
                  
\section{Introduction and main results}\label{sec:intro}

The presence of quasilinear elliptic equations with measure data $-\mathrm{div}(A(x,\nabla u))=\mu$ and their regularity results were early studied by P. Benilan {\it et al.} in~\cite{Benilan1975}, L. Boccardo {\it et al.} in~\cite{Boccardo1992, Boccardo1996}. In recent years, many different models arising in physics, chemistry, biology and various scientific fields that involve the measure data problems were also the subject of many researchers. Later, in this class of elliptic problems, there has been a growing interest in the study of existence, uniqueness and/or some properties of solutions. For instance, many notions of solutions (very weak solutions, entropy solutions, renormalized solutions and SOLA-Solutions Obtained by Limits of Approximations) were introduced in~\cite{Serrin, Lions, Murat1993, Maso1999}, and theories of regularity were early studied by F. Murat {\it et al.} in~\cite{Murat1993, Maso1999}, M. F. Betta {\it et al.} in~\cite{Betta1994, Betta1998} and L. Boccardo {\it et al.} in~\cite{Boccardo1992,Boccardo1996} etc. Afterwards, the development of solution estimates in weak Lebesgue and Sobolev spaces has been extensively studied in recent years, therein Lorentz spaces, Lorentz-Morrey and Orlicz spaces (weighted or non-weighted) could be listed. Lately, there have been many regularity results concerning to nonlinear elliptic equations with measure data have been further obtained. We refer the interested reader to~\cite{Mi1, Mi3, Min1, Mi4, KMin2013,KMi1, KMi2} for potential estimates, to~\cite{Jia2007, BYZ2008, Chle2018,NP-JFA-21,NP-20} for some estimates in Orlicz spaces and to~\cite{55QH4, MPT2018, MPT2019, MPTN2019, MPT9, MPT7} for gradient estimates in Lorentz and Lorentz-Morrey spaces under various assumptions on domain $\Omega$ and the range of $p$. 

1.1. \textbf{Elliptic equations with measure data.} Before diving into the motivation of our study, it is important to start with the problem description and some of its assumptions. Let us now consider the quasilinear elliptic equations in the presence of measure as following
\begin{align} 
\label{eq:main}
-\mathrm{div}(A(x,\nabla u)) = \mu  \ \mbox{ in } \ \Omega;  \quad
u = 0  \ \mbox{ on } \  \partial \Omega,
\end{align}
where $\Omega \subset \mathbb{R}^n$ is a bounded open subset, $n \ge 2$; the datum measure $\mu$ is defined in $\mathfrak{M}_b(\Omega)$-the space of all Radon measures on $\Omega$ with bounded total variation; the operator $A: \Omega \times \mathbb{R}^n \to \mathbb{R}^n$ here is a Carath\'eodory function satisfying the certain growth and monotonicity conditions. This operator and its properties are described in Section~\ref{sec:pre}.

There has been a lot of attention devoted to the notion of solutions to problem~\eqref{eq:main} in years. Naturally, one can start with the distributional solution (see Definition~\ref{def:weaksol}). However, the difficulty comes from the uniqueness of solution, and Serrin in \cite{Serrin} pointed out a counterexample where the uniqueness of distributional solution fails in general. Due to that reason, further notions of ``very weak'' solution have been investigated, such as SOLA or solutions in entropic and renormalized sense, etc. Inspired by preceding papers~\cite{55QH2, 55QH4, MPT2018, MPT2019} and for the present purpose, the solution to our problem~\eqref{eq:main} is set in the context of \emph{renormalized solutions}, that will be armed with a precise definition in the next section. Moreover, in the present work, domain assumptions specify the domain $\Omega$ has its complement being uniformly $p$-capacity thick. One notices that this class of domains is relatively large (including all domains with Lipschitz boundaries). A precise definition of the domain will be given in Section~\ref{sec:pre}. Otherwise, in order to obtain global regularity estimates over domains with non-smooth boundaries, it emphasizes that we consider the weaker assumption of domain $\Omega$ than that of the Reifenberg flatness assumption studied in various prior studies, see~\cite{ER60, DT1999, KT2003, MT2010}.

1.2. \textbf{Motivation.} There have been long-standing contributions in the research of regularity of solutions to the class of nonlinear elliptic equations with measure data. For instance, firstly by Boccardo {\it et al.} in~\cite{Boccardo1992, Boccardo1996}, under the assumption that $\Omega$ is bounded, when $\mu \in L^m(\Omega)$ for $1\le m<\frac{np}{n(p-1)+p}$, authors proved that the unique solution $u$ belongs to $W_0^{1,\frac{nm(p-1)}{n-m}}(\Omega)$. Later, the borderline case when $m=\frac{np}{n(p-1)+p}$ ($p<n$) was derived locally by Mingione in~\cite{Mi3}. Since then, whenever $2-\frac{1}{n}<p \le 2$, several local Calder\'on-Zygmund type estimates for the elliptic problems~\eqref{eq:main} were developed in different works~\cite{KMi1, KMi2, Mi1, Mi3, Mi4,Duzamin2,55DuzaMing} via potential estimates. Particularly, authors therein proposed the local estimates at least for the case $2 \le p \le n$, and further extension to global estimates has also been obtained by using maximal functions. 

More recently, under the $p$-capacity assumption on the complement of $\Omega$ (boundary is thick enough to satisfy a uniform density condition) for the regular case $2-\frac{1}{n}<p \le n$, the global gradient estimates of solutions to problem~\eqref{eq:main} have been obtained in the Lorentz setting due to~\cite{Phuc2014}; and it enables author further to treat the global regularity estimates in Lorentz-Morrey spaces in supercritical case, see~\cite{55Ph3}. According to the range value of $p$, our previous works in~\cite{MPT2018, MPTN2019} studied the gradient regularity to equation~\eqref{eq:main} in the framework of Lorentz spaces $L^{s,t}(\Omega)$, and  in Lorentz-Morrey spaces $L^{s,t;\kappa}(\Omega)$  respectively, for $0<s<p+\varepsilon$ remains bounded, $0 < t\le \infty$ and $0 \le \kappa<n$. 

In order to achieve better results where $L^{s,t}$ and $L^{s,t;\kappa}(\Omega)$ estimates are obtained for all $0<s<\infty$, it requires some additional structural assumptions of the problem setting. For instance, domain $\Omega$ is assumed under Reifenberg flatness condition (This class of domains include all $C^1$-domains, Lipschitz domains with small Lipschitz constants, and domains with fractal boundaries), and more information to operator $A$ that $A(x,\xi)$ is continuously differentiable in $\xi$ away from the origin and satisfies the smallness condition of BMO type. According to these assumptions, Nguyen {\it et al.} have got better Lorentz $L^{s,t}(\Omega)$ estimates to~\eqref{eq:main} in~\cite{55QH4} for singular case (that is, $\frac{3n-2}{2n-1} < p \le 2-\frac{1}{n}$). Further on, there have been numerous studies over Reifenberg domain for divergence type elliptic equations. For related results, we refer to~\cite{SSB2, BW1_1, CM2014, Phuc2015, CoMi2016, BCDKS, BW1, KZ, MPT7,NP-NA-20} and the reader can consult also the references therein.

Our study is motivated by the question raised in some recent advances, that how to treat such regularity results for the `very singular' problem when $1< p \le \frac{3n-2}{2n-1}$. There are two main results obtained in this work, where we prove the Lorentz and Lorentz-Morrey global estimates for the gradient of solutions in terms of fractional maximal functions. 

1.3. \textbf{Techniques.} Before delving into main results of this paper, we review various techniques concerning the regularity theory for elliptic and parabolic partial differential equations. From a very first approach proposed by Caffarelli and Peral \cite{CP1998} to obtain the local $W^{1,p}$ estimates for a class of $p$-Laplace equations, a new technique has been presented by Acerbi and Mingione in \cite{AM2007} to prove Calder\'on-Zygmund estimates with the use of $C^{0,1}$ estimates and maximal functions only. Later, this approach has been widely used and yielded very valuable results  in regularity theory. It is also noteworthy that there have been many other approaches for the regularity theory, such as method of using Hodge decomposition theorem by Iwaniec and Sbordone in \cite{IS1994}; method that is based on the truncation of certain maximal operators by Lewis in \cite{Lewis93}; method of using Stampacchia interpolation theorem in \cite{KS1980}; Dong, Kim and Krylov' results involving the sharp and maximal functions due to Fefferman-Stein theorem \cite{KS1980_1}, or some global Calder\'on-Zygmund estimates were first presented by Byun and Wang from geometrical approach \cite{BW2}, etc. Our approach in this paper is an improvement of the techniques introduced in \cite{Mi3,AM2007,BW2}, where the use of \emph{fractional-maximal distribution functions} (FMDs) is taken advantage to conclude global gradient estimates of solutions to~\eqref{eq:main}. In our efforts, it is hoped that interested readers will pay attention to \cite{NP-JFA-21} for further details.

1.4. \textbf{Main results.} Let us now state here the major results, via four important theorems as below. Throughout the paper, we always assume that complement of the bounded domain $\Omega$ satisfies the $p$-capacity uniform thickness condition with positive constants $c_0$ and $r_0$ (see Definition \ref{def:pcapathick}). Moreover, for simplicity and conciseness of notations, we introduce
\begin{align}\label{eq:mstar}
m^* = \frac{n}{2(p-1)n+2-p} \ \mbox{ and } \ m^{**} = \frac{pn}{(p-1)n+p}
\end{align} 
in main theorems and their proofs. Our method is to establish the level-set inequalities on distribution function as in Theorem~\ref{theo:lambda_estimateA} and~\ref{theo:good-lambda-2B}.

\begin{taggedtheorem}{A}
\label{theo:lambda_estimateA}
Let $1<p\le\frac{3n-2}{2n-1}$ and $\mu \in \mathfrak{M}_b(\Omega) \cap L^{m}(\Omega)$ for some $m \in (m^*,m^{**})$. Let $u$ be a renormalized solution to~\eqref{eq:main} and $\frac{n}{2n-1} < q < \frac{nm(p-1)}{n-m}$. Then one can find some constants $a = a(n,p,\Lambda,c_0) \in (0,1)$, $b = b(n,p,q,\Lambda,c_0) \in \mathbb{R}$, $\varepsilon_0 = \varepsilon_0(n,p,q,\Lambda,c_0)>0$ and $C = C(n,p,q,\Lambda,c_0,\mathrm{diam}(\Omega)/{r_0})>0$ such that the following inequality
\begin{align}\label{eq:main-dist}
d_{\mathbb{U}}(\varepsilon^{-a}\lambda) \le C \varepsilon d_{\mathbb{U}}(\lambda) + d_{\Pi}(\varepsilon^{b}\lambda)
\end{align}
holds for all $\lambda>0$ and $\varepsilon \in (0,\varepsilon_0)$. Here, the following notations are in use
\begin{align}\label{def:U-Pi}
\mathbb{U} = (\mathbf{M}(|\nabla u|^{q}))^{\frac{1}{q}}, \quad \Pi = (\mathbf{M}_{m}(|\mu|^{m}))^{\frac{1}{m(p-1)}},
\end{align} 
and the distribution function $d_{g}$ of a measurable function $g$ on $\Omega$ is defined as follows
\begin{align}\notag
d_{g}(\lambda) = \mathcal{L}^n \left(\left\{x \in \Omega: \ |g(x)|> \lambda\right\}\right), \quad \lambda \ge 0.
\end{align}
\end{taggedtheorem}

Throughout this paper, we use the abbreviation $\mathcal{L}^n(D)$ for the $n$-dimensional Lebesgue measure of a set $D \subset \mathbb{R}^n$. Moreover, for the sake of convenience,  the set $\{x \in \Omega: \, |g(x)| > \lambda\}$ may be simply written by $\{|g|>\lambda\}$.

As we mentioned above, the value of $m^{**}$ is prescribed to guarantee the existence of \emph{renormalized solution} to equation~\eqref{eq:main}. In addition, to conclude the Lorentz and Lorentz-Morrey gradient estimates,  it is critically important for us to construct the level-set inequalities, in which some comparison estimates in both the interior and on the boundary of domain have been effectively employed. In our analysis, results obtained in this paper are comparable with those in~\cite[Lemma 2.2]{55QH4} when $\frac{3n-2}{2n-1}<p \le 2-\frac{1}{n}$. However, in our work, an additional assumption on the datum $\mu$ is imposed, for which $\mu$ is a function belonging to $L^m(\Omega)$ for $m>1$. More precisely, we will focus our attention on datum $\mu$ assumed to be a function belonging to $L^m(\Omega)$, for $m>m^*$ specified above in Theorem~\ref{theo:lambda_estimateA}. Indeed, we notice that
\begin{align*}
p \le \frac{3n-2}{2n-1} \Longleftrightarrow m^* \ge 1,
\end{align*}
and in this case, generally speaking, based on methods given in~\cite{Benilan1995, 55QH4}, we cannot expect to obtain comparison results for a general measure datum $\mu$ (as a function in $L^1$) and the range of $p$; $p<\frac{3n-2}{2n-1}$. Otherwise, in \cite{55QH4}, authors dealt with the case $p>\frac{3n-2}{2n-1}$ (implies further that $m^*<1$), and this suffices to prove  both interior and boundary estimates with $\mu \in L^1(\Omega)$ (generally a measure). Therefore, it is natural to expect a more appropriate method or to require that one or more assumptions on our initial data. 

The next theorem~\ref{theo:good-lambda-2B} constructs the local version of level-set inequality presented in Theorem~\ref{theo:lambda_estimateA}, that can be applied to prove Lorentz-Morrey gradient estimates in Theorem~\ref{theo:mainD} later. However, the proof of Theorem~\ref{theo:good-lambda-2B} is rather similar to that of Theorem~\ref{theo:lambda_estimateA}, except for some estimates.

\begin{taggedtheorem}{B}
\label{theo:good-lambda-2B}
Let $1<p\le\frac{3n-2}{2n-1}$, $\mu \in \mathfrak{M}_b(\Omega) \cap L^{m}(\Omega)$ for some $m \in (m^*,m^{**})$. Let $u$ be a renormalized solution to~\eqref{eq:main} and $\frac{n}{2n-1} < q < \frac{nm(p-1)}{n-m}$. One finds some constants $a = a(n,p,\Lambda,c_0) \in (0,1)$, $b = b(n,p,q,\Lambda,c_0) \in \mathbb{R}$, $\varepsilon_0 = \varepsilon_0(n,p,q,\Lambda,c_0)>0$ and $C = C(n,p,q,\Lambda,c_0,\mathrm{diam}(\Omega)/{r_0})>0$ such that and for every $x \in \Omega$ and $0<\rho <\mathrm{diam}(\Omega)$ the following inequality
\begin{align}\label{eq:main-dist-2}
d_{\tilde{\mathbb{U}}}(B_{\rho}(x);\varepsilon^{-a}\lambda) \le C \varepsilon d_{\tilde{\mathbb{U}}}(B_{\rho}(x);\lambda) + d_{\tilde{\Pi}}(B_{\rho}(x);\varepsilon^{b}\lambda)
\end{align}
holds for any $\lambda>\varepsilon^{-b} \rho^{-\frac{n}{q}} \|\nabla u\|_{L^{q}(B_{10\rho}(x)\cap \Omega)}$ and $\varepsilon\in (0,\varepsilon_0)$, where 
\begin{align*}
\tilde{\mathbb{U}} = (\mathbf{M}(\chi_{B_{10\rho}(x)}|\nabla u|^{q}))^{\frac{1}{q}}, \quad \tilde{\Pi} = (\mathbf{M}_{m}(\chi_{B_{10\rho}(x)}|\mu|^{m}))^{\frac{1}{m(p-1)}},
\end{align*}
and the local distribution function $d_{g}(B_{\rho}(x);\cdot)$ of a measurable function $g$ is defined by
\begin{align}\notag
d_{g}(B_{\rho}(x);\lambda) = \mathcal{L}^n \left(\left\{\xi \in B_{\rho}(x)\cap \Omega: \ |g(\xi)|> \lambda\right\}\right), \quad \lambda \ge 0.
\end{align}
\end{taggedtheorem}

In Theorems~\ref{theo:mainC} and~\ref{theo:mainD}, some improved results of Lorentz and Morrey-Lorentz gradient estimates are given. This work extends our earlier works in~\cite{MPT2018, MPTN2019} when $\Omega$ satisfies the $p$-capacity condition and $p$ is singular. It also remarks that for specific case when $m\equiv 1$, the statements in Theorem~\ref{theo:mainC} and~\ref{theo:mainD} still hold, but they only make sense for a certain range of $p$, i.e. $\frac{3n-2}{2n-1}< p \le 2- \frac{1}{n}$. A detailed explanation will be discussed in Section~\ref{sec:pre}.

\begin{taggedtheorem}{C}
\label{theo:mainC}
Let $1<p\le\frac{3n-2}{2n-1}$ and $\mu \in \mathfrak{M}_b(\Omega)$. Assume that the given measure data $\mu \in L^{m}(\Omega)$ for some $m \in (m^*,m^{**})$. Then there exists a constant $\Theta_0 = \Theta_0(n,p,\Lambda,c_0) >p$ such that for any renormalized solution $u$ to~\eqref{eq:main}, there holds
\begin{align}  \label{eq:theo-mainC}
\|\nabla u\|_{L^{s,t}(\Omega)} \le C \|\left(\mathbf{M}_{m}(|\mu|^{m})\right)^{\frac{1}{m(p-1)}}\|_{L^{s,t}(\Omega)},
\end{align}
for any $0<s<\Theta_0$ and $0<t\le \infty$. The constant $C$ depends on $n$, $p$, $\Lambda$, $m$, $s$, $t$, $c_0$ and $D_0/r_0$.
\end{taggedtheorem}

\begin{taggedtheorem}{D}
\label{theo:mainD} 
Let $1<p\le\frac{3n-2}{2n-1}$ and $\mu \in \mathfrak{M}_b(\Omega) \cap L^{m}(\Omega)$ for some $m \in (m^*,m^{**})$. Then, there exist $\Theta_0 = \Theta_0(n,p,\Lambda,c_0) >p$, $\beta_0 = \beta_0(n,p,\Lambda) \in (0,1/2]$ such that for any renormalized solution $u$ to~\eqref{eq:main} with given datum $\mu \in L^{ms,mt; \kappa}(\Omega)$ satisfying $0< t \le \infty$, $0< \kappa < \min\left\{\frac{(n-m)\Theta_0}{m(p-1)}, n\right\}$, and
\begin{align}\label{est:cond-s}
\frac{\kappa}{n} \le s < \min\left\{\frac{\kappa}{m+m(1-\beta_0)(p-1)}, \frac{\Theta_0 \kappa}{m\Theta_0 + m\kappa(p-1)}\right\},
\end{align}
there holds
\begin{align}\label{est:LM}
\||\nabla u|^{p-1}\|_{L^{\frac{ m s \kappa}{\kappa-ms},\frac{m t \kappa}{\kappa-mt}; \kappa}(\Omega)} \le C \||\mu|^m\|_{L^{s,t;\kappa}(\Omega)}^{\frac{1}{m}}.
\end{align}
Here, the positive constant $C$ depends on $n$, $m$, $p$, $\Lambda$, $s$, $t$, $\kappa$, $c_0$ and $\mathrm{diam}(\Omega)/r_0$.
\end{taggedtheorem}

1.5. \textbf{The layout of the paper.} Let us briefly summarize the contents of the paper as follows. We begin with a few preliminaries about notation, definitions and assumptions of the problem in Section \ref{sec:pre}. Next, in Section~\ref{sec:prelem} we prove some preparatory lemmas including comparison estimates that are vitally important for the main theorems. It is worth emphasizing that most of comparison results regarding to problem~\eqref{eq:main} have been formulated and proved  up to the boundary. Section \ref{sec:level-set} is concerned with the level-set inequalities on distribution functions. Here we can succinctly present the use of FMDs to prove the good-$\lambda$ type inequalities in Theorems \ref{theo:lambda_estimateA} and \ref{theo:good-lambda-2B}. It is worth noting that the language of FMDs forms a key tool in our arguments in this paper. We end up with Section~\ref{sec:LorMorrey-res} for the proofs of Theorems \ref{theo:mainC} and \ref{theo:mainD}, in which some preparatory results from previous sections are combined to prove the global Lorentz and Lorentz-Morrey gradient norm regularity.

\section{Preliminaries}\label{sec:pre}

This section will introduce some convenient notations, assumptions and formulation of the problem~\eqref{eq:main}; review a number of the most important definitions and collect some additional preliminary results that related to our study in this paper. For further details, the interested reader can also consult the literature through the mentioned citations therein.
\subsection{Notation}
For notational simplicity, in the entirety of the paper we shall regard $\Omega \subset \mathbb{R}^n$ an open bounded domain, for $n \ge 2$. Generic positive constant will be denoted with a special letter $C$. The exact value of $C$ is unimportant, it may vary from  one occurence to the next and we still denote $C$ during chains of estimates. We note that the dependencies of constants $C$ on parameters will be clarified using parentheses. Further, we also employ specific constants with $C_1, C_2,$ et cetera, when needed. In the context, as usual, we write $B_\rho(x_0)$ for the ball of center $x_0$ and radius $\rho>0$. For $\mathcal{D} \subset \mathbb{R}^n$ being a measurable subset, let $h \in L^1_{\mathrm{loc}}(\mathbb{R}^n)$ be a measurable map, the integral average of $h$ over $\mathcal{D}$ will be denoted by
\begin{align*}
\fint_\mathcal{D}{h(x)dx} = \frac{1}{\mathcal{L}^n(\mathcal{D})}\int_\mathcal{D}{h(x)dx}.
\end{align*}
Moreover, the diameter of $\Omega$ will be denoted by $\mathrm{diam}(\Omega)$ for short, defined as
\begin{align*}
D_0 = \mathrm{diam}(\Omega) = \sup_{z_1,z_2 \in \Omega}{|z_1-z_2|}.
\end{align*}

\subsection{General assumptions}
\textbf{Assumption on domain (H1).} In this paper, our domain $\Omega \subset \mathbb{R}^n$ is assumed to satisfy the uniform capacity density condition. More precisely, the complement $\Omega^c:=\mathbb{R}^n \setminus \Omega$ satisfies the $p$-capacity uniform thickness. To our knowledge, this condition is very important for the existence of a solution to our problem and for a higher integrability property of the gradient. Let us now recall the definition of such domain as follows.
\begin{definition}[Domain with $p$-capacity condition]
\label{def:pcapathick}
We say that the complement set of $\Omega$ in $\mathbb{R}^n$, denoted $\Omega^c$, satisfies the \emph{uniformly $p$-thick condition} if there exist two constants $c_0,r_0>0$ such that
\begin{align}
\label{eq:capuni}
\mathrm{cap}_p(\Omega^c \cap \overline{B}_r(\zeta), B_{2r}(\zeta)) \ge c_0 \mathrm{cap}_p(\overline{B}_r(\zeta),B_{2r}(\zeta)),
\end{align}
for any $\zeta \in \Omega^c$ and $0<r \le r_0$.
\end{definition}
For the readers' convenience, we also include here the definition $p$-capacity of a set. 
\begin{definition}[$p$-capacity]
\label{def:pcapaset}
Let $p>1$ and $Q$ be a compact subset of $\Omega$, we define the $p$-capacity of $Q$ by:
\begin{align*}
\mathrm{cap}_p(Q,\Omega) = \inf_{\psi \in K}  \int_\Omega{|\nabla \psi|^p dx}, 
\end{align*}
where $K = \left\{\psi \in C_c^\infty(\Omega), \, \psi \ge \chi_Q \right\}$ and $\chi_Q$ denotes the characteristic function of $Q$.
\end{definition}
The definition of $p$-capacity can also be extended to capacity of any open set $\mathbb{O} \subseteq \Omega$ by 
\begin{align*}
\mathrm{cap}_p(\mathbb{O},\Omega) = \sup_{Q \subseteq \mathbb{O}, \, Q \, \text{compact}} \mathrm{cap}_p(Q,\Omega).
\end{align*}
Finally, for an any subset $E \subseteq \Omega$, one defines the $p$-capacity of $E$  by:
\begin{align*}
\mathrm{cap}_p(E,\Omega) = \inf_{\mathbb{O} \subseteq E, \, \mathbb{O} \, \text{open}} \mathrm{cap}_p(\mathbb{O},\Omega).
\end{align*}
As far as we know, the class of domains satisfy definition of $p$-capacity uniform thickness includes all domains with Lipschitz boundaries or  satisfy a uniform corkscrew condition (see~\cite{HKM1993} for further reading), and~\eqref{eq:capuni} still remains valid for balls centered outside a uniformly $p$-thick domain. Moreover, the uniform $p$-capacity is necessary for the validity of Poincar\'e's and Sobolev's inequalities, that are very helpful in our proofs later. For further properties of the $p$-capacity condition, we refer to the books~\cite[Chater 2]{Mazya1985} and~\cite[Chapter 2]{HKM1993}.

\textbf{Assumption on the nonlinearity (H2).} As usual, in the setting of equation~$-\mathrm{div}(A(x,\nabla u))=\mu$, the nonlinear operator $A: \Omega \times \mathbb{R}^n \rightarrow \mathbb{R}^n$ is a Carath\'eodory function. Moreover, there exist $1<p\le n$ and two constants $\Lambda_1$, $\Lambda_2>0$ such that 
\begin{align}\label{eq:A1}
\left| A(x,\zeta) \right| &\le \Lambda_1 |\zeta|^{p-1}, \tag{A1}
\end{align}
and
\begin{align}
\label{eq:A2}
\langle A(x,\zeta)-A(x,\eta), \zeta - \eta \rangle &\ge \Lambda_2 \left( |\zeta| + |\eta| \right)^{p-2}|\zeta - \eta|^2 \tag{A2}
\end{align}
holds for a.e. $x \in \Omega$ and every $(\zeta,\eta) \in \mathbb{R}^n \times \mathbb{R}^n \setminus \{(0,0)\}$.
From above conditions \eqref{eq:A1} and \eqref{eq:A2}, one recognizes that the operator $A$ is defined on $W_0^{1,p}(\Omega)$ with values in its dual space $W^{-1,p'}(\Omega)$ ($p'$ is the H\"older conjugate of $p$). 

\textbf{Assumption on $p$ (H3).} As aforementioned, in this paper, the growth exponent $p$ is considered as a real number in following range
\begin{align}\label{ass:p}
1<p \le \displaystyle{\frac{3n-2}{2n-1}}.
\end{align}

\textbf{Assumption on the measure datum $\mu$ (H4).} Let us consider $\mathfrak{M}_b(\Omega)$ as the set of all Radon measures with bounded total variation on $\Omega$ and $C_b(\Omega)$ as the set of all bounded, continuous functions on $\Omega$, such that $\int_\Omega{\psi d\mu}<+\infty$ for all $\psi \in C_b(\Omega)$ and $\mu \in \mathfrak{M}_b(\Omega)$. One denotes by $\mu^+$, $\mu^-$ and $|\mu|$ respectively, the positive part, negative part and the total variation of a measure $\mu$ in $\mathfrak{M}_b(\Omega)$.
\begin{definition}
\label{def:convergeMes}
We say that $(\mu_n)_n$ converges to $\mu$ in $\mathfrak{M}_b(\Omega)$ in a narrow topology if 
\begin{align}
\label{eq:convergeMes}
\lim_{n \to +\infty}{\int_\Omega{\psi d\mu_n}} = \int_\Omega{\psi d\mu},
\end{align}
for every $\psi \in C_b(\Omega)$. 
\end{definition}
One defines $\mathfrak{M}_0(\Omega)$ as the set of $\mu \in \mathfrak{M}_b(\Omega)$ such that $\mu(Q)=0$ for every Borel set $Q \subseteq \Omega$ with $\mathrm{cap}_p(Q,\Omega)=0$. We also define by $\mathfrak{M}_s(\Omega)$ the set of $\mu \in\mathfrak{M}_b(\Omega)$ for which one can find a Borel set $E \subset \Omega$ with $\mathrm{cap}_p(E,\Omega)=0$ such that $\mu(Q) = \mu(E \cap Q)$ for any Borel set $Q \subseteq \Omega$.
\begin{remark}
\label{rem:pairmeas}
For any $\mu \in \mathfrak{M}_b(\Omega)$, one can find a unique pair $(\mu_0, \mu_s)$ such that $\mu = \mu_0+\mu_s$, with $\mu_0 \in \mathfrak{M}_0(\Omega)$ and $\mu_s \in \mathfrak{M}_s(\Omega)$. Moreover if $\mu$ is nonnegative, then $\mu_0$ and $\mu_s$ are also nonnegative. Hence, the measures $\mu_s$ and $\mu_0$ will be called the \emph{singular} and the \emph{absolutely continuous} part of $\mu$.
\end{remark}

Throughout this paper, we shall work with the renormalized solution, where the datum $\mu$ satisfies assumption (H4). However, we firstly recall the definition of weak solutions in sense of distributions. 
\begin{definition}[Distributional solution]\label{def:weaksol}
We say that $w \in W^{1,p}_0(\Omega)$ is a weak solution to equation~\eqref{eq:main} if the variational formula
\begin{align*}
\int_{\Omega} \langle A(x,\nabla w), \nabla \psi\rangle dx = \int_{\Omega} \langle \mu, \psi\rangle dx,
\end{align*}
is valid for every $\psi \in W^{1,p}_0(\Omega)$.
\end{definition}
The problem~\eqref{eq:main} does not admit a weak solution under these above assumptions. However, in general, ones can expect to establish a notion of weak solution such that we can prove the existence and uniqueness of such solution. The concept of renormalized solution, was first introduced by R.~J. DiPerna \emph{et al.} in~\cite{DLion1989} when studying of the Boltzmann equation, and then adapted to nonlinear elliptic problems with Dirichlet boundary conditions by L. Boccardo \emph{et al.} in~\cite{Boccardo1989}. An equivalent notion of solutions, called entropy solution, was then introduced by P. B\'enilan \emph{et al.} in~\cite{Benilan1995}.

Let us also recall the definition of \emph{renormalized solution}, that was early studied in~\cite{BMMP, Boccardo1996, Maso1999}. To do this, we first introduce the truncature operator. For $k>0$, we follow the notation of \emph{truncation} operator at level $\pm k$, that is $T_k:\mathbb{R} \to \mathbb{R}$ defined by
\begin{align}\notag 
T_k (s) = \max\left\{ -k; \min\{s,k\}\right\}, \quad  s \in \mathbb{R},
\end{align}
that belongs to $W_0^{1,p}(\Omega)$, which satisfies $-\mathrm{div} A(x,\nabla T_k(u)) = \mu_k$ in the sense of distribution for a finite measure $\mu_k$ in $\Omega$. 
\begin{definition}\label{def:truncature}
Let $u$ be a given measurable function that is finite almost everywhere in $\Omega$, and satisfies $T_k(u) \in W^{1,1}_0(\Omega)$ for all $k>0$. Then, there is a unique measurable function $v: \Omega \to \mathbb{R}^n$ satisfying:
\begin{align}\notag
\nabla T_k(u) = \chi_{\{|u| \le k\}} v , \quad \text{a.e. in} \ \ \Omega, \quad \text{for any} \ k>0.
\end{align}
\end{definition}
Function $v$ here is so-called ``generalized distributional gradient'' of $u$ and in this paper it is still written by $\nabla u$ when no ambiguity will arise. One notices that if $u \in L^1_{loc}(\Omega)$, this function differs from the distributional gradient of $u$, and it coincides exactly with the usual gradient for every $u \in W^{1,1}(\Omega)$.

There are several equivalent definitions of \emph{renormalized solutions} (see \cite{Maso1997, Maso1999} and related references), here we will use the following notion of ``very weak'' solution.

\begin{definition}[Renormalized solution]
\label{def:renormsol3}
Let $\mu = \mu_0+\mu_s \in \mathfrak{M}_b(\Omega)$ with $\mu_0 \in \mathfrak{M}_0(\Omega)$ and $\mu_s \in \mathfrak{M}_s(\Omega)$ and $u$ be a finite measurable function defined in $\Omega$. We say that $u$ is a renormalized solution of~\eqref{eq:main} if $u$ has the following properties
\begin{itemize}
\item $T_k(u) \in W^{1,p}_0(\Omega)$ for every $k>0$, $|{\nabla u}|\in L^{q(p-1)}(\Omega)$ for any $0<q<\frac{n}{n-1}$; 
\item for every $k>0$ one can find nonnegative Radon measures $\delta_k^+, \, \delta_k^- \in\mathfrak{M}_0(\Omega)$ concentrated on $\{u=k\}$ and $\{u=-k\}$,  such that $\delta_k^+\rightarrow\mu_s^+$, $\delta_k^-\rightarrow\mu_s^-$ in the narrow topology of measures and that
 \begin{align*}
 \int_{\{|u|<k\}}\langle A(x,\nabla u),\nabla \psi\rangle
  	dx=\int_{\{|u|<k\}}{\psi d}{\mu_{0}}+\int_{\Omega}\psi d\delta_k^+ -\int_{\Omega}\psi d\delta_k^-,
 \end{align*}
  	for any $\psi\in L^{\infty}(\Omega) \cap W^{1,p}_0(\Omega)$.
\end{itemize}
\end{definition}
The following Lemma~\ref{lem:nablau} were given in~\cite[Theorem 4.1]{Maso1999}, that characterizes the classical global Lebesgue gradient estimate for solution $u$ to~\eqref{eq:main} with given measure data $\mu \in L^1(\Omega)$ and the convergence result. And we refer the reader to~\cite{Benilan1995, Boccardo1992, Boccardo1996} for the proofs.
\begin{lemma}\label{lem:nablau}
Let $u$ be a renormalized solution to~\eqref{eq:main} with a given finite measure data $\mu$ in $\Omega$. Then one can find $C=C(n,p)>0$ such that 
\begin{align}\label{eq:nablau}
\|\nabla u\|_{L^{\tilde{p},\infty}(\Omega)}\leq C\left[|\mu|(\Omega)\right]^{\frac{1}{p-1}}, \quad \mbox{ with } \ \tilde{p} = \frac{(p-1)n}{n-1}.
\end{align}
\end{lemma}

\begin{proposition}\label{prop:nablauconverge}
Let $\mu \in L^1(\Omega)$ and $u$ in $L^s(\Omega)$ the renormalized solution to~\eqref{eq:main} with measure data $\mu$. For every $k \in \mathbb{N}$, let $u_k$ be the renormalized solution of~\eqref{eq:main} with data $\mu_k \in L^{\frac{p}{p-1}}(\Omega)$ such that $\mu_k \to \mu$ weakly in $L^1(\Omega)$. Then, one can find a subsequence $\{u_{k_j}\}_{j}$ of $\{u_k\}_k$ such that 
\begin{align*}
\begin{cases} u_{k_j} \to u \mbox{ in } L^q(\Omega), \ \mbox{ for any } 0< q < \frac{(p-1)n}{n-p}, \\ \nabla u_{k_j} \to \nabla u \mbox{ in } L^q(\Omega), \ \mbox{ for any } 0<q<\frac{(p-1)n}{n-1}. \end{cases}
\end{align*}
\end{proposition}
\subsection{Function spaces}
\begin{definition}[Distribution function]\label{def:dist}
For every measurable function $g$ on $\Omega$ and $Q \subset \mathbb{R}^n$, the distribution function $d_g(Q;\cdot)$ of $g$ is defined in $\mathbb{R}^+$ as follows
\begin{align}\label{def-Df-2}
d_g(Q;\lambda) := \mathcal{L}^n \left(\left\{x \in Q \cap \Omega: \ |g(x)|> \lambda\right\}\right), \quad \lambda \ge 0.
\end{align}
If $\Omega \subset Q$, we write $d_g(\lambda)$ instead of $d_g(\Omega;\lambda)$ for simplicity.
\end{definition}
Using this definition of the distribution function one can rewrite the definition of Lorentz spaces as below.
\begin{definition}[Lorentz spaces]
For $0<s<\infty$ and $0<t\le \infty$, we denote by $L^{s,t}(\Omega)$ (see~\cite{55Gra}) the Lorentz space containing all Lebesgue measurable functions $g$ such that
\begin{align}
\label{eq:lorentz}
\|g\|_{L^{s,t}(\Omega)} = \left[s\int_0^\infty{\lambda^{t-1} [d_g(\lambda)]^{\frac{t}{s}} d\lambda} \right]^{\frac{1}{t}} < +\infty.
\end{align} 
\label{def:Lorentz}
\end{definition}
When $t = \infty$, $L^{s,\infty}(\Omega)$ is known as the usual weak-$L^s$ space or Marcinkiewicz space with the following quasinorm
\begin{align}\notag
\|g\|_{L^{s,\infty}(\Omega)} = \sup_{\lambda>0}{\lambda [d_g(\lambda)]^{\frac{1}{s}}}.
\end{align}
Cavalieri's principle shows that if $s=t$, then the Lorentz space $L^{s,s}(\Omega)$ becomes the usual Lebesgue space $L^s(\Omega)$. More precisely, the spaces are nested increasingly with respect to the second parameter $t$:
\begin{align*}
L^{s,1}(\Omega) \subset L^{s,t}(\Omega) \subset L^{s,\infty}(\Omega).
\end{align*}
\begin{definition}[Lorentz-Morrey spaces]
\label{def:Lorentz-Morrey}
 For $0<s<\infty$, $0<t \le \infty$ and for some $0<\kappa \le n$, the Lorentz-Morrey space $L^{s,t;\kappa}(\Omega)$ contains all functions $g \in L^{s,t}(\Omega)$ such that
\begin{align}\label{eq:LMsp}
\|g\|_{L^{s,t;\kappa}(\Omega)}=\sup_{\varrho\in(0,D_0); \, x \in \Omega}{\varrho^{\frac{\kappa-n}{s}}}\|g\|_{L^{s,t}(\Omega \cap B_\varrho(x))} < +\infty.
\end{align}
In the case $\kappa = n$ the space $L^{s,t;n}(\Omega)$ coincides exactly with the Lorentz space $L^{s,t}(\Omega)$.
\end{definition}
\begin{definition}[Fractional maximal functions, \cite{K1997}]
\label{def:Malpha}
The fractional maximal function $\mathbf{M}_\alpha$ for $0 \le \alpha \le n$ is defined by:
\begin{align}\label{eq:Malpha}
\mathbf{M}_\alpha g(x) = \sup_{\rho>0}{\rho^\alpha \fint_{B_\rho(x)}{|g(y)|dy}},~~ \forall x \in \mathbb{R}^n, \quad \mbox{for } g \in L_{loc}^1(\mathbb{R}^n).
\end{align}
\end{definition}
For the case $\alpha=0$, one obtains the Hardy-Littlewood maximal function, $\mathbf{M}g = \mathbf{M}_0g$ for $\mbox{for } g \in L_{loc}^1(\mathbb{R}^n)$ defined  by:
\begin{align}\label{eq:M0}
\mathbf{M}g(x) = \sup_{\rho>0}{\fint_{B_{\rho}(x)}|g(y)|dy},~~ \forall x \in \mathbb{R}^n.
\end{align}
The cut-off function of $\mathbf{M}$ at level $r>0$ will be denoted by
\begin{align*}
\mathbf{M}^rg(x) = \sup_{0<\rho<r}{\fint_{B_{\rho}(x)}|g(y)|dy},~~ \forall x \in \mathbb{R}^n.
\end{align*}
The well-known result of maximal operator is the bounded property on $L^p(\mathbb{R}^n)$ for $1< p \le \infty$, that is there exists $C=C(n,p)>0$ such that:
\begin{align*}
\|\mathbf{M}g\|_{L^p(\mathbb{R}^n)} \le C \|g\|_{L^p(\mathbb{R}^n)}, \quad \forall g \in L^p(\mathbb{R}^n).
\end{align*}
Moreover, one knows that $\mathbf{M}$ is \emph{weak-type (1,1)}, it means that there is  $C=C(n)>0$ such that 
\begin{align*}
d_{\mathbf{M}g}(\lambda) \le \frac{C}{\lambda} \|g\|_{L^1(\mathbb{R}^n)}, \quad \forall \lambda>0, \ g \in L^1(\mathbb{R}^n).
\end{align*}
The standard and classical references can be found in many places such as \cite{55Gra}. Here we recall the useful bounded properties that we use in the present paper.
\begin{lemma}[\cite{55Gra}]
\label{lem:boundM}
The maximal operator $\mathbf{M}$ is bounded from $L^s(\mathbb{R}^n)$ to $L^{s,\infty}(\mathbb{R}^n)$ for every $s \ge 1$, this means
\begin{align}\notag
d_{\mathbf{M}g}(\lambda) \le \frac{C}{\lambda^s}\int_{\mathbb{R}^n}{|g(x)|^s dx}, \quad \mbox{ for all } \lambda>0.
\end{align}
\end{lemma}
\begin{lemma}[\cite{55Gra}]
\label{lem:boundMlorentz}
The operator $\mathbf{M}$ is bounded in the Lorentz space $L^{s,t}(\mathbb{R}^n)$, for $s>1$ and $0<t\le \infty$, this means
\begin{align}\notag
\|\mathbf{M}g\|_{L^{s,t}(\mathbb{R}^n)} \le C \|g\|_{L^{s,t}(\mathbb{R}^n)}.
\end{align}
\end{lemma}

\section{Preliminary technical lemmas}
\label{sec:prelem}
The purpose of this section is to construct, state and prove some technical lemmas that will be necessary to use later. Furthermore, a series of comparison procedures between solutions of~\eqref{eq:main} in $\Omega$ and the ones of homogeneous equations in arbitrary balls, that are very important to obtain our main results.

In the remaining parts of this paper, we always mention $u$-the solution of our problem~\eqref{eq:main}, the \emph{renormalized solution} in which the existence and uniqueness always make sense. In addition, for the sake of simplicity, we assume in useful lemmas below that domain $\Omega$ satisfies the hypothesis (H1), under the assumption (H3) and two parameters $m^*$ and $m^{**}$ are clarified as in~\eqref{eq:mstar}.
\subsection{Preparatory lemmas}
\begin{lemma}\label{lem:grad-u}
Let $\mu \in L^{m}(\Omega)$ for some $m \in (1,m^{**})$ and $u$ be a renormalized solution to~\eqref{eq:main}. Then $\nabla u \in L^{\frac{nm (p-1)}{n-m}}(\Omega)$ and there exists a positive constant $C$ such that
\begin{align}\label{eq:grad-u}
\|\nabla u\|_{L^{\frac{nm(p-1)}{n-m}}(\Omega)}\leq C \|\mu\|_{L^{m}(\Omega)}^{\frac{1}{p-1}}.
\end{align}
\end{lemma}

\begin{lemma}\label{lem:conver-grad-u}
Let $u$ be the renormalized solution to~\eqref{eq:main} with measure data $\mu \in L^m(\Omega)$ for some  $m \in (1,m^{**})$. Assume that a sequence $\{u_k\}_k$ is the renormalized solution to~\eqref{eq:main} with data $\mu_k \in L^{\frac{p}{m(p-1)}}(\Omega)$ such that $\mu_k$ converges to $\mu$ weakly in $L^m(\Omega)$. Then, there exists a subsequence $\{u_{k_j}\}_{j}$ such that $u_{k_j} \to u$ and $\nabla u_{k_j} \to \nabla u$ in $L^q(\Omega)$ for any $0< q < \frac{nm (p-1)}{n-m}$.
\end{lemma}
These obtained results remain valid in the case of measure data $\mu \in L^1(\Omega)$, i.e.~$m\equiv 1$ and they are exact the ones concluded in Lemma \ref{lem:nablau} and Proposition \ref{prop:nablauconverge}, respectively. We next derive local interior and boundary estimates in terms of given data $\mu$ in following lemmas, they are independent of the above ones.
\subsection{Interior gradient estimates}
Our attention has been first driven to the interior estimates. We first fix a point $x_0 \in \Omega$, for $0<2R \le r_0$ ($r_0$ is the constant given in \eqref{eq:capuni}). For each ball $B_R := B_R(x_0) \subset \Omega$ and $u \in W_{\mathrm{loc}}^{1,p}(\Omega)$, we consider $w \in W_{0}^{1,p}(B_R) +u$, unique solution to the homogeneous problem as follows
\begin{align} \label{eq:interior-w}
-\mathrm{div}(A(x,\nabla w)) = 0  \ \mbox{ in } \ B_R; \quad
w = u \  \mbox{ on } \ \partial B_{R}.
\end{align}
Here, for the convenience of the readers, a version of interior Gehring's lemma applied to solution $w$ of \eqref{eq:interior-w} is restated. This is also known as the ``reverse'' H\"older integral inequality with increasing supports (see \cite{Gehring}, \cite[Theorem 6.7]{Giu}). Technique of using this inequality with small exponents to gradient estimates was proposed by G. Mingione \textit{et al.} in \cite{Mi1} and along with it, many research approaches have since been developed.

\begin{lemma}\label{lem:rev-Holder}
Let $u \in W_{\mathrm{loc}}^{1,p}(\Omega)$ be a solution to equation~\eqref{eq:main} and $w \in W_{0}^{1,p}(B_R) +u$ be the unique solution to~\eqref{eq:interior-w}. There exists a constant $\Theta_0 = \Theta_0(n,p,\Lambda,c_0)>p$ such that
\begin{align}\label{ineq:rev-Holder}
\left(\fint_{B_{r}(y)} |\nabla w|^{\Theta_0}dx\right)^{\frac{1}{\Theta_0}} \le C \left(\fint_{B_{2r}(y)}|\nabla w|^{p-1}dx\right)^{\frac{1}{p-1}},
\end{align}
for all $B_{2r}(y) \subset B_{R}$ and $C=C(n, p, \Lambda)$.
\end{lemma}
We next perform a comparison gradient result for solutions to both problems \eqref{eq:main} and \eqref{eq:interior-w}. The conclusion of following result is a modified version of a result has been proved in \cite{55QH4} when $\mu \in L^1(\Omega)$ and $\frac{3n-2}{2n-1}<p \le 2-\frac{1}{n}$. Our approach is based on methods given in~\cite{Benilan1995, 55QH4}, where the utilization of H\"older's and Sobolev's inequalities are very important to conclude some estimates. However, the method cannot be directly applied itself to obtain comparison results for a general measure datum $\mu$ (as a function in $L^1$) and the range of $p$; i.e. $p<\frac{3n-2}{2n-1}$. In our discussion, one comes to expect the more appropriate method or the requirement that one or more assumptions on initial data. And in this study, to overcome the difficulty of the very singular $p$, authors impose a stronger assumption on datum $\mu$, to be a function belonging to $L^m(\Omega)$, with $m>1$ will be specified later in the proof. It is worth emphasizing that, the statement and proof of comparison estimates for the case $1<p \le \frac{3n-2}{2n-1}$ are not new. Here, for the convenience of reader, we restate and only present the sketch of proof. 

Let us state that important result via following lemma and here, it is noticeable that we derive the local estimates in the ball $B_R$.

\begin{lemma}\label{lem:int-compare}
Let $1  < p \le \frac{3n-2}{2n-1}$ and $\mu \in L^{m}(B_R)$ for some $m \in (m^*, n)$. Assume that $u \in W_{\mathrm{loc}}^{1,p}(\Omega)$ be a solution to equation~\eqref{eq:main}  and $w \in W_{0}^{1,p}(B_R) +u$ be the unique solution to equation~\eqref{eq:interior-w}. For any $q$ satisfying the following condition 
\begin{align}\label{cond:range-q}
  \frac{n}{2n-1} < q < \frac{nm(p-1)}{n-m},
\end{align}
there exists a constant $C>0$ only depending on $n$, $p$, $q$ and $m$ such that
\begin{align}\label{ineq:int-compare}
\left(\fint_{B_R} |\nabla u - \nabla w|^{q}dx\right)^{\frac{1}{q}} \le C \mathcal{F}_R(\mu)^{\frac{1}{p-1}} + C \mathcal{F}_R(\mu) \left(\fint_{B_R}|\nabla u|^{q}dx\right)^{\frac{2-p}{q}},
\end{align} 
where the function $\mathcal{F}_R$ is defined by
\begin{align}\label{eq:def-F}
\mathcal{F}_R(\mu) = \left(R^{m}\fint_{B_R}|\mu|^{m}dx\right)^{\frac{1}{m}}.
\end{align}
\end{lemma}
\begin{proof} The proof of this lemma can be divided into three steps. In the first step, we apply the H{\"o}lder and Sobolev inequality to obtain the following estimate
\begin{align}\label{est:Im-4}
\left(\int_{B_R} |\nabla v|^{q}dx\right)^{\frac{1}{q}} \le C \left( \int_{B_R}|\nabla v| |v|^{-\frac{s}{q}}dx\right)^{\frac{q}{q-s}},
\end{align}
where $v := u - w$, and $0<s<q<1$ satisfying some conditions such that the Sobolev inequality is valid. For $1<p<2$, let us consider the function $\mathcal{G}(u,w) := (|\nabla u|^2 + |\nabla w|^2)^{\frac{p-2}{2}}|\nabla u - \nabla w|^2$. Then we apply the following fundamental inequality
\begin{align*}
|\nabla v|  \le C \left(\mathcal{G}(u,w)^{\frac{1}{p}} + \mathcal{G}(u,w)^{\frac{1}{2}} |\nabla u|^{\frac{2-p}{2}}\right),
\end{align*}
on the right-hand side of~\eqref{est:Im-4} to arrive that
\begin{align}\label{est:Im-4f}
\left(\int_{B_R} |\nabla v|^{q}dx\right)^{\frac{1}{q}} \le C \left[\int_{B_R}\left(|v|^{-\frac{s}{q}} \mathcal{G}(u,w)^{\frac{1}{p}} + |v|^{-\frac{s}{q}} \mathcal{G}(u,w)^{\frac{1}{2}} |\nabla u|^{\frac{2-p}{2}}\right)dx\right]^{\frac{q}{q-s}}.
\end{align}

In the second step we need to estimate the terms in the form $\int_{B_R} |v|^{-s_1} \mathcal{G}(u,w)^{s_2} dx$ appeared on the right-hand side of~\eqref{est:Im-4f}. To do this, we will apply the variational form of equations~\eqref{eq:main} and~\eqref{eq:interior-w} with a consistent test function. The key idea is to choose the test function by $\varphi = T_{k}^{\varepsilon}(|v|^{\alpha-1}v) \in W_0^{1,p}(B_R)$ for $\varepsilon, \, \alpha>0$, where 
\begin{align*}
T_{k}^{\varepsilon}(t) = (k+\varepsilon) \, \mathrm{sign}(t) \max\left\{0,\min\left\{1,\frac{|t|-\varepsilon}{k}\right\}\right\}, \quad k>0.
\end{align*}
Applying techniques related to the method in~\cite[Lemma 4.2]{Benilan1995}, or~\cite[Section 5]{Maso1999} or~\cite{55QH4}, we may prove that
\begin{align}\label{est:Im-1k}
\left(\int_{B_R}  |v|^{\eta(\alpha-1)} \mathcal{G}(u,w)^{\eta} dx\right)^{\frac{1}{\eta}} & \le C R^{\frac{n}{\eta}-\frac{n(\alpha+\beta-\gamma)}{\beta}}  \|\mu\|_{L^{m}(B_R)} \|v\|^{\alpha}_{L^{\beta}(B_R)},
\end{align}
where $\mathcal{F}_R(\mu)$ defined as in~\eqref{eq:def-F}, $\gamma \in (0,\alpha)$, $\beta = \frac{m\gamma}{m-1}$ and $0 < \eta < \frac{\beta}{\alpha+\beta-\gamma}$. Thanks to Sobolev's inequality, we imply from~\eqref{est:Im-1k} that
\begin{align}\label{est:Im-1}
\left(\int_{B_R}  |v|^{\eta(\alpha-1)} \mathcal{G}(u,w)^{\eta} dx\right)^{\frac{1}{\eta}} & \le C R^{n\left(\frac{1}{\eta}-\frac{\alpha}{\beta}\right)-1} \mathcal{F}_R(\mu) \left( \int_{B_R}|\nabla v| |v|^{-\frac{s}{q}}dx\right)^{\frac{q\alpha}{q-s}},
\end{align} 

In the last step we apply~\eqref{est:Im-1} and Young inequality with suitable value of parameters $s, \alpha, \beta, \gamma, \eta$ to obtain~\eqref{ineq:int-compare}. We remark that the assumption~\eqref{cond:range-q} comes from all sufficient conditions of appeared parameters here.
\end{proof}

The following result holds.

\begin{lemma}\label{lem:res3.4P}
Assume that $w$ is the solution to~\eqref{eq:interior-w}. There exist two constants $\beta_0 = \beta_0(n,p,\Lambda) \in (0,1/2]$ and $C_1 = C_1(n,p,\Lambda)>0$ such that 
\begin{align}\label{eq:nablaw_estimate}
\left(\fint_{B_\rho(y)}{|\nabla w|^p dx} \right)^{\frac{1}{p}} \le C_1 \left(\frac{\rho}{r} \right)^{\beta_0-1}\left(\fint_{B_r(y)}{|\nabla w|^p dx} \right)^{\frac{1}{p}},
\end{align}
for any $B_\rho(y) \subset B_r(y) \subset B_{R}$. Moreover, for any $s \in (0,p]$ one can find a constant $C_2 = C_2(n,p,\Lambda,s)$ such that
\begin{align}\label{eq:nablaw_estimate-2}
\left( \fint_{B_\rho(y)}{|\nabla w|^s dx} \right)^{\frac{1}{s}} \le C_2 \left(\frac{\rho}{r} \right)^{\beta_0-1} \left(\fint_{B_r(y)}{|\nabla w|^s dx} \right)^{\frac{1}{s}},
\end{align}
for any $B_\rho(y) \subset B_r(y) \subset B_{R}$.
\end{lemma}

\begin{proof}
The inequality~\eqref{eq:nablaw_estimate} comes from the standard interior H\"older continuity of solutions, its proof can be found in~\cite[Theorem 7.7]{Giu} and we do not write down all the details here. Applying  the inequality~\eqref{eq:nablaw_estimate} and Lemma~\ref{lem:rev-Holder}, one obtains~\eqref{eq:nablaw_estimate-2}.
\end{proof}

In order to prove the key lemma below, one refers to Lemma~\ref{lem:hanlin} in~\cite[Lemma 1.4]{Han-Lin} as follows.
\begin{lemma}\label{lem:hanlin}
Let $\eta \in (0,1)$, $0\le \beta<\alpha$ and $h: [0,R] \to [0,\infty)$ be a non-decreasing function. Suppose that 
$$h(\rho) \le P\left[\left(\frac{\rho}{r} \right)^\alpha + \varepsilon \right] h(r) + Q r^\beta,$$
for any $0<\rho \le \eta r <R$ and $\varepsilon>0$, with positive constants  $P$, $Q$. Then, there exists a constant $\varepsilon_0=\varepsilon_0(P,\alpha,\beta,\eta)$ such that if $\varepsilon \in (0,\varepsilon_0)$ then
$$h(\rho) \le C \left[\left(\frac{\rho}{r} \right)^\alpha h(r) + Q\rho^\beta \right],$$
for all $0<\rho \le r \le R$, where $C$ depends on $P,\alpha$ and $\beta$. 
\end{lemma}

\begin{lemma}\label{lem:res7.2H} Let $1  < p \le \frac{3n-2}{2n-1}$ and $u$ be a solution to equation~\eqref{eq:main} with  $\mu \in L^{m}(\Omega)$ for some $m \in (m^*, n)$.  Let $\beta_0 \in (0,1/2]$ be as in Lemma~\ref{lem:res3.4P}. Then, for any  $ q \in \left(  \frac{n}{2n-1}, \frac{nm(p-1)}{n-m}\right)$ and
$$m + m(p-1)(1-\beta_0)< \sigma \le n,$$ 
one can find a constant $C = C(n,p,\Lambda,c_0,\sigma)>0$ such that 
\begin{align}\label{eq:res7.2H}
\left(\int_{B_\rho(y)}{|\nabla u|^{q}dx}\right)^{\frac{1}{q}}\leq  C \rho^{\frac{n}{q}-\delta} \left[\mathbf{M}_{\sigma}^{D_0}(|\mu|^m)(y) \right]^{\frac{1}{m(p-1)}},
\end{align}
for all $B_\rho(y) \subset\subset \Omega$, where $\delta = \frac{\sigma-m}{m(p-1)}$.
\end{lemma}

\begin{proof}
Let us take $B_{r}(y) \subset\subset \Omega$ and $0<\rho\le r/2$. Applying Lemma~\ref{lem:int-compare} with $B_{R} = B_{r}(y)$, one gives:
\begin{align}\nonumber
\left(\fint_{B_{r}(y)}{|\nabla (u-w)|^{q}dx} \right)^{\frac{1}{q}}&\leq C \left(r^{m}\fint_{B_r(y)}|\mu|^{m}dx\right)^{\frac{1}{m(p-1)}} \\ \label{eq:bylem1}
& \qquad +	C \left(r^{m}\fint_{B_r(y)}|\mu|^{m}dx\right)^{\frac{1}{m}} \left(	\fint_{B_{r}(y)}|\nabla u|^{q}dx\right)^{\frac{2-p}{q}}.
\end{align}
Thanks to Lemma~\ref{lem:res3.4P} with $B_\rho(y) \subset B_{r}(y) \subset B_{R}$ and $s=q$, we obtain that
\begin{align}\label{eq:bylem2}
\left(\fint_{B_\rho(y)}{|\nabla w|^{q}} \right)^{\frac{1}{q}} \le C \left(\frac{\rho}{r} \right)^{\beta_0-1}\left(\fint_{B_{2r/3}(y)}{|\nabla w|^{q}} \right)^{\frac{1}{q}}.
\end{align}
Combining two inequalities from~\eqref{eq:bylem1} and~\eqref{eq:bylem2} with the fact that
\begin{align*}
 \int_{B_{2r/3}(y)}{|\nabla w|^{q}}dx \le C \int_{B_{r}(y)}{|\nabla u|^{q}}dx,
\end{align*}
one obtains
\begin{align}\nonumber
\left( \int_{B_\rho(y)}{|\nabla u|^{q}dx}\right)^{\frac{1}{q}} &\leq \left(\int_{B_\rho(y)}{|\nabla w|^{q} dx} \right)^{\frac{1}{q}} + \left(\int_{B_\rho(y)}{|\nabla u-\nabla w|^{q}dx} \right)^{\frac{1}{q}}\\ \nonumber
&\leq C\left(\frac{\rho}{r} \right)^{\frac{n}{q}+\beta_0-1} \left(\int_{B_{r(y)}}{|\nabla u|^{q}dx} \right)^{\frac{1}{q}} \\  \nonumber
& \qquad + C \rho^{\frac{n}{q}} \left(r^{m}\fint_{B_r(y)}|\mu|^{m}dx\right)^{\frac{1}{m(p-1)}} \\ \label{est:rho-1}
& \qquad + C\rho^{\frac{n(p-1)}{q}} \left(r^{m}\fint_{B_r(y)}|\mu|^{m}dx\right)^{\frac{1}{m}} \left(\frac{\rho}{r}\right)^{\frac{n(2-p)}{q}}\left(\int_{B_{r(y)}}{|\nabla u|^{q}}dx \right)^{\frac{2-p}{q}}.
\end{align}
For any $\varepsilon \in (0,1)$, using Young's inequality for the last term in~\eqref{est:rho-1} and notice that $\rho < r$, one finds 
\begin{align}\nonumber
\left( \int_{B_\rho(y)}{|\nabla u|^{q}dx}\right)^{\frac{1}{q}} &\leq C\left(\frac{\rho}{r} \right)^{\frac{n}{q}+\beta_0-1} \left(\int_{B_{r(y)}}{|\nabla u|^{q}dx} \right)^{\frac{1}{q}} \\ \nonumber
& \qquad  + \varepsilon \left(\frac{\rho}{r}\right)^{\frac{n}{q}}\left(\int_{B_{r(y)}}{|\nabla u|^{q}}dx \right)^{\frac{1}{q}} + C_\varepsilon \rho^{\frac{n}{q}} \left(r^{m}\fint_{B_r(y)}|\mu|^{m}dx\right)^{\frac{1}{m(p-1)}} \\ \label{eq:btholder}
& \le C \left[ \left(\frac{\rho}{r} \right)^{\frac{n}{q}+\beta_0-1} + \varepsilon\right] \left(\int_{B_{r(y)}}{|\nabla u|^{q}dx} \right)^{\frac{1}{q}} + C r^{\frac{n}{q}} \left(r^{m}\fint_{B_r(y)}|\mu|^{m}dx\right)^{\frac{1}{m(p-1)}}.
\end{align}
For any $m + m(p-1)(1-\beta_0)< \sigma \le n$, let us set  $\delta = \frac{\sigma-m}{m(p-1)}$, the inequality~\eqref{eq:btholder} can be rewritten as follows
\begin{align*}
h(\rho) &\leq C\left[\left(\frac{\rho}{r} \right)^{\frac{n}{q}+\beta_0-1} +\varepsilon\right]h(r) + C r^{\frac{n}{q}-\delta}\left(r^{m(p-1)\delta + m - n} \int_{B_{r}(y)}|\mu|^{m}dx \right)^{\frac{1}{m(p-1)}} \\
& \le C\left[\left(\frac{\rho}{r} \right)^{\frac{n}{q}+\beta_0-1} +\varepsilon\right]h(r) + C r^{\frac{n}{q}-\delta}  \left( \mathbf{M}_{\sigma}^{D_0} (|\mu|^{m}) \right)^{\frac{1}{m(p-1)}},
\end{align*}
where the function $h: \ [0,D_0] \to [0, \infty)$ defined by
\begin{align}\label{eq:Phi}
h(\varrho)=\left( \int_{B_{\varrho}(y)}{|\nabla u|^{q}dx}\right)^{\frac{1}{q}}, \qquad \varrho >0.
\end{align}
Applying Lemma~\ref{lem:hanlin} with $\alpha = \frac{n}{q}+\beta_0-1$, $\beta = \frac{n}{q}-\delta$ and $Q =  C\left( \mathbf{M}_{\sigma}^{D_0} (|\mu|^{m}) \right)^{\frac{1}{m(p-1)}}$, there is $\varepsilon_0>0$ such that for every  $\varepsilon \in (0,\varepsilon_0)$ and $0<\rho<r\le D_0$ there holds
\begin{align*}
h(\rho) \le C \left[\left(\frac{\rho}{r} \right)^{\frac{n}{q}-\delta} h(r) + \rho^{\frac{n}{q}-\delta}  \left( \mathbf{M}_{\sigma}^{D_0} (|\mu|^{m}) \right)^{\frac{1}{m(p-1)}} \right],
\end{align*}
and one thus gets
\begin{align}\label{eq:inMtheta}
\left( \int_{B_\rho(y)}{|\nabla u|^{q}dx}\right)^{\frac{1}{q}} \leq  C \rho^{\frac{n}{q}-\delta} \left[\left(\frac{1}{D_0} \right)^{\frac{n}{q}-\delta} \left( \int_{\Omega}{|\nabla u|^{q}dx}\right)^{q} + \left( \mathbf{M}_{\sigma}^{D_0} (|\mu|^{m}) \right)^{\frac{1}{m(p-1)}}\right] .
\end{align}
According to H{\"o}lder's inequality and~\eqref{eq:grad-u} in Lemma~\ref{lem:grad-u}, it gives that
\begin{align}\nonumber
\left(\frac{1}{D_0} \right)^{\frac{n}{q}-\delta} \left( \int_{\Omega}{|\nabla u|^{q}dx}\right)^{\frac{1}{q}} & \le C \left(\frac{1}{D_0} \right)^{\frac{n}{q}-\delta} D_0^{\frac{n}{q}-\frac{n-m}{m(p-1)}} \left(\int_{\Omega} |\nabla u|^{\frac{nm(p-1)}{n-m}} dx\right)^{\frac{n-m}{nm(p-1)}} \\ \nonumber
& \le C D_0^{\frac{\sigma -n}{m(p-1)}} \left(\int_{\Omega} |\mu|^mdx\right)^{\frac{1}{m(p-1)}} \\ \label{eq:2inM}
& \leq C  \left( \mathbf{M}_{\sigma}^{D_0} (|\mu|^{m}) \right)^{\frac{1}{m(p-1)}}.
\end{align}
From \eqref{eq:inMtheta} and \eqref{eq:2inM} we may conclude that~\eqref{eq:res7.2H} holds.
\end{proof}
\subsection{Boundary gradient estimates}
In the remaining part of this section, we are able to deal with some up-to-boundary comparison estimates of the domain. Under the hypothesis that $\mathbb{R}^n\setminus \Omega$ is uniformly $p$-thickness with constants $c_0$, $r_0$ as in \eqref{eq:capuni}, it is possible to prove estimates similar to what obtained in the interior of domain. 

Let $x_0 \in \partial \Omega$ be a boundary point and $R \in (0,r_0/2)$, we set $\Omega_{2R} = B_{2R}(x_0) \cap \Omega$. 
For any $u \in W_{\mathrm{loc}}^{1,p}(\Omega)$ being a solution to equation~\eqref{eq:main}, we again consider $w \in W_{0}^{1,p}(\Omega_{2R}) +u$ that is a unique solution to the following reference problem
\begin{align} \label{eq:boundary-w}
-\mathrm{div}(A(x,\nabla w)) = 0  \ \mbox{ in } \ \Omega_{2R}; \quad
w = u \  \mbox{ on } \ \partial \Omega_{2R}.
\end{align}

\begin{lemma}\label{lem:rev-Holder-boundary}
Let $u \in W_{\mathrm{loc}}^{1,p}(\Omega)$ be a solution to equation~\eqref{eq:main} and $w \in W_{0}^{1,p}(\Omega_{2R}) +u$ be the unique solution to~\eqref{eq:boundary-w}. One can find a constant $\Theta_0=\Theta_0(n,p,\Lambda,c_0)>p$ such that
\begin{align}\label{ineq:rev-Holder-boundary}
\left(\fint_{B_{r}(y)} |\nabla w|^{\Theta_0}dx\right)^{\frac{1}{\Theta_0}} \le C \left(\fint_{B_{2r}(y)}|\nabla w|^{p-1}dx\right)^{\frac{1}{p-1}},
\end{align}
for all $B_{2r}(y) \subset B_{2R}(x_0)$ and $C$ depends on $n, p, \Lambda$.
\end{lemma}

We remark that in several articles such as~\cite{55QH4}, the reverse H{\"o}lder are usually stated as the following form
\begin{align*}
\left(\fint_{B_{r/2}(y)} |\nabla w|^{\theta}dx\right)^{\frac{1}{\theta}} \le C \left(\fint_{B_{3r}(y)}|\nabla w|^{p-1}dx\right)^{\frac{1}{p-1}},
\end{align*}
for all $B_{3r}(y) \subset B_{2R}(x_0)$. However, we can prove~\eqref{ineq:rev-Holder-boundary} by the same argument as in \cite[Lemma 3.5]{MPT2018}. We next state the counterpart of Lemmas~\ref{lem:int-compare} up to the boundary of the domain $\Omega$.

\begin{lemma}\label{lem:boundary-compare}
Let $1  < p \le \frac{3n-2}{2n-1}$ and $\mu \in L^{m}(B_{2R})$ for some $m \in (m^*,n)$. Assume that $u \in W_{\mathrm{loc}}^{1,p}(\Omega)$ be a solution to equation~\eqref{eq:main} and $w \in W_{0}^{1,p}(\Omega_{2R}) +u$ be the unique solution to equation~\eqref{eq:boundary-w}. For any $q$ satisfying the following condition 
\begin{align*}
  \frac{n}{2n-1} < q < \frac{nm(p-1)}{n-m},
\end{align*}
there exists a constant $C>0$ only depending on $n$, $p$, $q$ and $m$ such that
\begin{align*}
\left(\fint_{B_{2R}(x_0)} |\nabla u - \nabla w|^{q}dx\right)^{\frac{1}{q}} \le  C & \left(R^{m}\fint_{B_{2R}(x_0)}|\mu|^{m}dx\right)^{\frac{1}{m(p-1)}} \\ &  + C \left(R^{m}\fint_{B_{2R}(x_0)}|\mu|^{m}dx\right)^{\frac{1}{m}} \left(\fint_{B_{2R}(x_0)}|\nabla u|^{q}dx\right)^{\frac{2-p}{q}}.
\end{align*} 
\end{lemma}

\begin{lemma}\label{lem:boundary-res3.4P}
Let $w$ be solution to~\eqref{eq:boundary-w}. One finds two constants $\beta_0 = \beta_0(n,p,\Lambda) \in (0,1/2]$ and $C_1 = C_1(n,p,\Lambda)>0$ such that 
\begin{align}\label{eq:nablaw_boundary}
\left(\fint_{B_\rho(y)}{|\nabla w|^p dx} \right)^{\frac{1}{p}} \le C_1 \left(\frac{\rho}{r} \right)^{\beta_0-1}\left(\fint_{B_r(y)}{|\nabla w|^p dx} \right)^{\frac{1}{p}},
\end{align}
for any $B_\rho(y) \subset B_r(y) \subset B_{2R}(x_0)$. Moreover, for any $s \in (0,p]$ there exists a positive constant $C_2 = C_2(n,p,\Lambda,s)$ such that
\begin{align}\label{eq:nablaw_boundary-2}
\left( \fint_{B_\rho(y)}{|\nabla w|^s dx} \right)^{\frac{1}{s}} \le C_2 \left(\frac{\rho}{r} \right)^{\beta_0-1} \left(\fint_{B_r(y)}{|\nabla w|^s dx} \right)^{\frac{1}{s}},
\end{align}
for any $B_\rho(y) \subset B_r(y) \subset B_{2R}(x_0)$.
\end{lemma}
We now state the selection Lemma~\ref{lem:boundary-res7.2H}, which establishes $L^q$-estimate for gradient of solution $u$ up to the boundary. The proof of such result is very similar and follows the argument of Lemma~\ref{lem:res7.2H}.

\begin{lemma}\label{lem:boundary-res7.2H}
Let $1  < p \le \frac{3n-2}{2n-1}$ and $u$ be a solution to equation~\eqref{eq:main} with  $\mu \in L^{m}(\Omega)$ for some $m \in (m^*, n)$. Let $\beta_0 \in (0,1/2]$ be as in Lemma~\ref{lem:boundary-res3.4P}. Then, for any 
$$m + m(p-1)(1-\beta_0)< \sigma \le n,$$ 
there exists $C = C(n,p,\Lambda,c_0,\sigma)>0$ such that
\begin{align}\label{eq:res7.2H-boundary}
\left(\int_{B_\rho(y)}{|\nabla u|^{q}dx}\right)^{\frac{1}{q}}\leq  C \rho^{\frac{n}{q}-\delta} \left[\mathbf{M}_{\sigma}^{D_0}(|\mu|^m)(y) \right]^{\frac{1}{m(p-1)}},
\end{align}
for all $B_\rho(y) \cap \partial \Omega \neq \emptyset$, where $\delta = \frac{\sigma-m}{m(p-1)}$.
\end{lemma}

\section{Global gradient estimates}
\label{sec:LorMorrey-res}

In this section, we first establish several level-set inequalities related to the distribution function of measurable functions (Definition~\ref{def:dist}), which were considered in our previous works such as~\cite{NP-JFA-21,NP-20,PNB-21}. Then the proofs of Theorem~\ref{theo:mainC} and~\ref{theo:mainD} are shown to give the global Lorentz and Lorentz-Morrey gradient estimates, respectively. To do this, we mainly use the level-set inequalities in Theorem~\ref{theo:lambda_estimateA} and~\ref{theo:good-lambda-2B}. It is worth mentioning that these results are connected to our previous ones in~\cite{MPT2018, MPTN2019}, but for $1<p \le \frac{3n-2}{2n-1}$.

\subsection{Level-set inequalities on distribution functions}\label{sec:level-set}

We assume that $u$ is a renormalized solution to~\eqref{eq:main} with data $\mu \in \mathfrak{M}_b(\Omega) \cap L^{m}(\Omega)$ for some $m \in (m^*,m^{**})$ and a given parameter $\frac{n}{2n-1} < q < \frac{nm(p-1)}{n-m}$. Let us recall the notations in~\eqref{def:U-Pi} that we use in this section
\begin{align*}
\mathbb{U}(\zeta) := (\mathbf{M}(|\nabla u|^{q})(\zeta))^{\frac{1}{q}}, \quad \Pi(\zeta) := (\mathbf{M}_{m}(|\mu|^{m})(\zeta))^{\frac{1}{m(p-1)}}, \quad \zeta \in \mathbb{R}^n.
\end{align*} 
The cut-off function related to $\mathbb{U}$ is defined by 
$$\mathbb{U}_{2r}^x(\zeta) := (\mathbf{M}^{2r}(\chi_{B_{2r}(x)}|\nabla u|^{q})(\zeta))^{\frac{1}{q}},  \quad \zeta \in \mathbb{R}^n.$$

\begin{lemma}\label{lem:N1}
For every $a >0$, one can find $\varepsilon_0 = \varepsilon_0(a,n,p,q,\Lambda,D_0)>0$ small enough and $b = b(a,q) \in \mathbb{R}$ such that if provided $\xi_0 \in \Omega$ satisfying $\Pi(\xi_0) \le \varepsilon^{b}\lambda$ then there holds
\begin{equation}\label{est:lem-N1}
d_{\mathbb{U}}(\varepsilon^{-a}\lambda) \leq \varepsilon D_0^n,  \quad \forall \varepsilon \in (0,\varepsilon_0).
\end{equation}
\end{lemma}
\begin{proof}
Using the boundedness of the maximal function $\mathbf{M}$ from Lebesgue space $L^1(\mathbb{R}^n)$ into Marcinkiewicz space $L^{1,\infty}(\mathbb{R}^n)$ in Lemma~\ref{lem:boundM} and H{\"o}lder's inequality,  one obtains that
\begin{align}\label{eq:est100}
d_{\mathbb{U}}(\varepsilon^{-a}\lambda) & \le \frac{C}{(\varepsilon^{-a}\lambda)^q} \int_{\Omega} |\nabla u|^{q}dx \le  \frac{C D_0^{n-\frac{q(n-m)}{m(p-1)}}}{(\varepsilon^{-a}\lambda)^q}  \left(\int_{\Omega}|\nabla u|^{\frac{nm(p-1)}{n-m}}\right)^{\frac{q(n-m)}{nm(p-1)}},
\end{align}
with notice that $q < \frac{nm(p-1)}{n-m}$. On the other hand, the gradient bound of $u$ in Lemma~\ref{lem:grad-u} and the existence of $\xi_0 \in \Omega$ such that $\Pi(\xi_0) \le \varepsilon^{b}\lambda$ imply that
\begin{align}\label{eq:bt4}
\|\nabla u\|_{L^{\frac{nm(p-1)}{n-m}}(\Omega)} \le C \|\mu\|_{L^{m}(\Omega)}^{\frac{1}{p-1}} \le C \left[D_0^{\frac{n}{m}-1}(\varepsilon^b\lambda)^{p-1}\right]^{\frac{1}{p-1}}.
\end{align}
It follows from~\eqref{eq:bt4} and~\eqref{eq:est100} that
\begin{align}\label{est:i}
d_{\mathbb{U}}(\varepsilon^{-a}\lambda) &\leq \frac{C }{(\varepsilon^{-a}\lambda)^{q}} D_0^{n-\frac{q(n-m)}{m(p-1)}} \left[ D_0^{\frac{n}{m}-1}(\varepsilon^b\lambda)^{p-1}\right]^{\frac{q}{p-1}} \leq  C\varepsilon^{(a+b)q} D_0^n.
\end{align}
In order to obtain~\eqref{est:lem-N1}, we can take $b$ satisfying $(a+b)q > 1$ and $\varepsilon_0>0$ such that $C\varepsilon_0^{(a+b)q-1}<1$ in~\eqref{est:i}.
\end{proof}

\begin{lemma}\label{lem:N2}
For every $a >0$ and $x \in \Omega$, one can find $\varepsilon_0 = \varepsilon_0(a,n,q)>0$ small enough such that if provided $\xi_1 \in \Omega \cap B_r(x)$ satisfying $\mathbb{U}(\xi_1) \le \lambda$ then there holds
\begin{equation}\label{est:lem-N2}
d_{\mathbb{U}}(B_r(x);\varepsilon^{-a}\lambda) \leq d_{\mathbb{U}_{2r}^x}(B_r(x);\varepsilon^{-a}\lambda), \quad \forall \varepsilon \in (0,\varepsilon_0).
\end{equation}
\end{lemma}
\begin{proof}
For every $y \in B_r(x)$ and $\rho \ge r$, we note that 
$$B_{\rho}(y)\subset B_{\rho+r}(x)\subset B_{\rho+2r}(\xi_1)\subset B_{3\rho}(\xi_1),$$
which deduces to
\begin{align*}
\left(\sup_{\rho \ge r} \fint_{B_\rho(y)}|\nabla u|^{q}dx \right)^{\frac{1}{q}}  \le 3^{\frac{n}{q}} \left(\sup_{\rho \ge 3r} \fint_{B_\rho(\xi_1)}|\nabla u|^{q}dx \right)^{\frac{1}{q}}  \le 3^{\frac{n}{q}} \mathbb{U}(\xi_1) \le 3^{\frac{n}{q}} \lambda.
\end{align*}
Here the last inequality comes from the assumption $\mathbb{U}(\xi_1) \le \lambda$. Thus,
\begin{align*}
\mathbb{U}(y) & \le \max \left\{ \left( \sup_{0<\rho<r} \fint_{B_\rho(y)} \chi_{B_{2r}(x)}|\nabla u|^{q}dx\right)^{\frac{1}{q}} ; \ \left(\sup_{\rho \ge r} \fint_{B_\rho(y)}|\nabla u|^{q}dx \right)^{\frac{1}{q}}  \right\} \\
& \le \max \left\{\mathbb{U}_{2r}^x(y); \ 3^{\frac{n}{q}} \lambda \right\}, \quad \mbox{ for all } \ y\in B_r(x),
\end{align*}
which implies that $\{\mathbb{U}>3^{\frac{n}{q}}\lambda\} \cap B_r(x) = \emptyset$. For this reason, by choosing $\varepsilon_0 \in (0,1)$ such that $\varepsilon_0^{-a}>3^{\frac{n}{q}}$, we will get that
\begin{align}\notag
\left\{\mathbb{U}>\varepsilon^{-a}\lambda\right\} \cap B_r(x) \subset \left\{\mathbb{U}_{2r}^x>\varepsilon^{-a}\lambda\right\} \cap B_r(x), \quad \forall \varepsilon \in (0,\varepsilon_0),
\end{align}
which allows us to conclude~\eqref{est:lem-N2}.
\end{proof}

\begin{lemma}\label{lem:N3}
For every $\varepsilon \in (0,1)$, one can find constants $a = a(n,p,\Lambda)\in(0,1)$ and $b = b(a,q) \in \mathbb{R}$ such that if provided $\xi_2$, $\xi_3\in B_r(x) \cap \Omega$ satisfying
\begin{align}\label{eq:bt7}
\mathbb{U}(\xi_2) \leq \lambda \ \mbox{ and } \ \Pi(\xi_3) \le \varepsilon^{b}\lambda,
\end{align}	 	 
then there holds
\begin{equation}\label{est:lem-N3}
d_{\mathbb{U}_{2r}^x}(B_r(x);\varepsilon^{-a}\lambda) \le C \varepsilon r^n.
\end{equation}
\end{lemma}
\begin{proof}
Let $u_k\in W_0^{1,p}(\Omega)$ be the unique solution to the following homogeneous problem:
\begin{equation}\label{eq:u_k-sol}
\begin{cases} -\mathrm{div}(A(x,\nabla u_k)) & =  \mu_k   \quad \ \text{in} \ \Omega,\\ \hspace{2cm} {u}_{k} & =  0 \qquad   \text{on} \ \partial\Omega, \end{cases}
\end{equation}
where $\mu_k = T_k(\mu)$.  In order to prove~\eqref{est:lem-N3}, for the sake of clarity we will  consider two cases: $B_{4r}(x)\subset\subset \Omega$ and $B_{4r}(x)\cap \Omega^{c}\not=\emptyset$.\\

Let us consider the first case $B_{4r}(x)\subset\subset\Omega$. Applying Lemma~\ref{lem:int-compare} for $w_k$ being the unique solution to:
\begin{equation}\notag
\begin{cases} -\mathrm{div}(A(x,\nabla w_k)) &=0, \quad \ \, \text{in}\ \ B_{4r}(x), \\ \hspace{2cm} w_k &= u_k, \quad \text{on} \ \ \partial B_{4r}(x),\end{cases}
\end{equation}
with $\mu = \mu_k$ and $B_{R} = B_{4r}(x)$, one has a constant $C = C(n,p,\Lambda,q,m)>0$ such that:
\begin{align}\label{eq:btgeneral}
\left(\fint_{B_{4r}(x)}{|\nabla u_k - \nabla w_k|^{q}dx} \right)^{\frac{1}{q}} & \le C\left[\mathcal{F}_{4r}(\mu_k)\right]^{\frac{1}{p-1}} + C \mathcal{F}_{4r}(\mu_k) \left(\fint_{B_{4r}(x)}{|\nabla u_k|^{q}dx} \right)^{\frac{2-p}{q}},
\end{align}
where the function $\mathcal{F}_{4r}$ is defined by
$$\mathcal{F}_{4r}(\mu_k) = \left((4r)^{m}\fint_{B_{4r}(x)} |\mu_k|^{m}dx \right)^{\frac{1}{m}}.$$
Moreover, applying the reverse H{\"o}lder's inequality in Lemma~\ref{lem:rev-Holder}, there exists a constant $\Theta_0 >p$ such that
\begin{align}\nonumber
\left(\fint_{B_{2r}(x)}|\nabla w_k|^{\Theta_0} dx\right)^{\frac{1}{\Theta_0}}  &\leq C \left(\fint_{B_{4r}(x)}|\nabla w_k|^{p-1} dx\right)^{\frac{1}{p-1}} \\  \label{eq:bt17}
&\le C \left(\fint_{B_{4r}(x)}{|\nabla u_k|^{q}dx} \right)^{\frac{1}{q}}+  C\left( \fint_{B_{4r}(x)}{|\nabla u_k - \nabla w_k|^{q}dx} \right)^{\frac{1}{q}},
\end{align}
where the last inequality comes from the H{\"o}lder's inequality with notice that $q>p-1$.

On the other hand, it easy to see that the distribution function given in~\eqref{est:lem-N3} can be decomposed as 
\begin{align}\nonumber
d_{\mathbb{U}_{2r}^x}(B_r(x);\varepsilon^{-a}\lambda) & \le \mathcal{L}^n\left(\left\{\mathbf{M}^{2r}\left(\chi_{B_{2r}(x)}|\nabla (u-u_k)|^{q}\right)^{\frac{1}{q}}>3^{-\frac{1}{q}}\varepsilon^{- a}\lambda\right\}\cap B_r(x)\right) \\ \nonumber
& \quad +   \mathcal{L}^n\left(\left\{\mathbf{M}^{2r}\left(\chi_{B_{2r}(x)}|\nabla (u_k-w_k)|^{q}\right)^{\frac{1}{q}}>3^{-\frac{1}{q}}\varepsilon^{- a}\lambda\right\}\cap B_r(x)\right)  \\  \label{eq:bt11}	
& \qquad  +  \mathcal{L}^n\left(\left\{\mathbf{M}^{2r}\left(\chi_{B_{2r}(x)}|\nabla w_k|^{q}\right)^{\frac{1}{q}}>3^{-\frac{1}{q}}\varepsilon^{- a}\lambda\right\}\cap B_r(x)\right).  
\end{align}
Applying Lemma~\ref{lem:boundM} with $s = 1$ and $s = \frac{\Theta_0}{q}>1$ for three terms on the right hand side of~\eqref{eq:bt11} respectively, one obtains that
\begin{align}\nonumber
d_{\mathbb{U}_{2r}^x}(B_r(x);\varepsilon^{-a}\lambda) &\le \frac{Cr^n}{(\varepsilon^{-a}\lambda)^{q}} \fint_{B_{2r}(x)}{|\nabla u - \nabla u_k|^{q}dx}  \\ \nonumber
& \qquad + \frac{Cr^n}{(\varepsilon^{-a}\lambda)^{q}}   \fint_{B_{2r}(x)}{|\nabla u_k - \nabla w_k|^{q}dx} \\ \label{eq:bt12}   
& \qquad \qquad + \frac{Cr^n}{(\varepsilon^{-a}\lambda)^{\Theta_0}}\fint_{B_{2r}(x)}{|\nabla w_k|^{\Theta_0}dx}.
\end{align}
Substituting both estimates \eqref{eq:btgeneral} and \eqref{eq:bt17} into~\eqref{eq:bt12} we get
\begin{align*}
d_{\mathbb{U}_{2r}^x}(B_r(x);\varepsilon^{-a}\lambda) &\le  \frac{Cr^n}{(\varepsilon^{-a}\lambda)^{q}} \fint_{B_{4r}(x)}{|\nabla u - \nabla u_k|^{q}dx} \\
& \quad + \frac{Cr^n}{(\varepsilon^{-a}\lambda)^{q}} \left[ \mathcal{F}_{4r}(\mu_k)^{\frac{1}{p-1}}+ \mathcal{F}_{4r}(\mu_k) \left(\fint_{B_{4r}(x)}{|\nabla u_k|^{q}dx} \right)^{\frac{2-p}{q}} \right]^{q} \\ 
& \qquad + \frac{Cr^n}{(\varepsilon^{-a}\lambda)^{\Theta_0}} \left[\left(\fint_{B_{4r}(x)}{|\nabla u_k|^{q}dx} \right)^{\frac{1}{q}} +  \mathcal{F}_{4r}(\mu_k)^{\frac{1}{p-1}} \right. \\ & 
\left. \qquad \quad + \mathcal{F}_{4r}(\mu_k) \left(\fint_{B_{4r}(x)}{|\nabla u_k|^{q}dx} \right)^{\frac{2-p}{q}} \right]^{\Theta}.
\end{align*}
Passing $k \to \infty$ and applying Lemma~\ref{lem:conver-grad-u} the above inequality becomes
\begin{align}\nonumber
d_{\mathbb{U}_{2r}^x}(B_r(x);\varepsilon^{-a}\lambda) &\le  \frac{Cr^n}{(\varepsilon^{-a}\lambda)^{q}} \left[ \mathcal{F}_{4r}(\mu)^{\frac{1}{p-1}}+ \mathcal{F}_{4r}(\mu) \left(\fint_{B_{4r}(x)}{|\nabla u|^{q}dx} \right)^{\frac{2-p}{q}} \right]^{q} \\ \nonumber
& \quad + \frac{Cr^n}{(\varepsilon^{-a}\lambda)^{\Theta_0}} \left[\left(\fint_{B_{4r}(x)}{|\nabla u|^{q}dx} \right)^{\frac{1}{q}} +  \mathcal{F}_{4r}(\mu)^{\frac{1}{p-1}} \right. \\ \label{eq:bt20} & \left. \qquad + \mathcal{F}_{4r}(\mu) \left(\fint_{B_{4r}(x)}{|\nabla u|^{q}dx} \right)^{\frac{2-p}{q}} \right]^{\Theta}.
\end{align}	
Since $|x-\xi_2|<r$, it follows  $B_{4r}(x)\subset B_{5r}(\xi_2)$. Thus one obtains from assumption~\eqref{eq:bt7} that
\begin{align}\label{eq:btp1}
\left(\fint_{B_{4r}(x)}|\nabla u|^{q}dx\right)^{\frac{1}{q}} & \leq \left( \left(\frac{5}{4}\right)^n 	\fint_{B_{5r}(\xi_2)}|\nabla u|^{q}dx \right)^{\frac{1}{q}} \le C \mathbb{U}(\xi_2) \le C \lambda.
\end{align}
Similarly, from $|x-\xi_3|<r$, we have $B_{4r}(x)\subset B_{5r}(\xi_3)$ which gives us
\begin{align}\label{eq:btp2}
\mathcal{F}_{4r}(\mu) & \le  C\left(\frac{1}{(5r)^{n-m}}\int_{B_{5r}(\xi_3)}|\mu|^{m}dx\right)^{\frac{1}{m}} \le C \Pi(\xi_3) \le C (\varepsilon^{b}\lambda)^{p-1}.
\end{align}
Substituting~\eqref{eq:btp1} and~\eqref{eq:btp2} into~\eqref{eq:bt20}, we may conclude that
\begin{align}\nonumber
d_{\mathbb{U}_{2r}^x}(B_r(x);\varepsilon^{-a}\lambda) &\le  \frac{Cr^n}{(\varepsilon^{-a}\lambda)^{q}} \left[ \varepsilon^{b}\lambda + (\varepsilon^{b}\lambda)^{p-1} \lambda^{2-p} \right]^{q} \\ \nonumber
& \qquad \qquad  + \frac{Cr^n}{(\varepsilon^{-a}\lambda)^{\Theta_0}} \left[ \lambda +  \varepsilon^{b}\lambda + (\varepsilon^{b}\lambda)^{p-1} \lambda^{2-p}  \right]^{\Theta_0} \\ \label{eq:bt210}
 & \le C r^n  \left(\varepsilon^{q(a+b)} + \varepsilon^{q(a+b(p-1))}\right)  + C r^n \varepsilon^{a\Theta_0}. 
\end{align}		
We are reduced to proving~\eqref{est:lem-N3} by choosing $a$, $b$ in~\eqref{eq:bt210} such that $a\Theta_0=1$ and $(a+b(p-1))q=1$. Note that with this choice, we also have $a\in (0,1)$ and $(a+b)q>1$ which is assumed at~\eqref{est:i} in Lemma~\ref{lem:N2}.\\

We next consider the second case $B_{4r}(x) \cap \Omega^c \neq \emptyset$. Let $\xi_4 \in \partial\Omega$ such that 
$$|\xi_4-x|=\text{dist}(x,\partial\Omega)\leq 4r.$$ 
It is easy to see that $B_{4r}(x) \subset B_{8r}(\xi_4)$. Applying Lemma \ref{lem:boundary-compare} with $v_k$ being the solution to:
\begin{equation}\notag
\begin{cases} -\mathrm{div}(A(x,\nabla v_k) &=0, \quad \ \, \text{in}\ \ B_{8r}(\xi_4),\\
\hspace{2cm} v_k &= u_k, \quad \text{on} \ \ \partial B_{8r}(\xi_4), \end{cases}
\end{equation}
for $\mu = \mu_k$ and $B_{2R} = B_{8r}(\xi_4)$ and $u_k \in W^{1,p}_{0}(\Omega)$ being the solution to~\eqref{eq:u_k-sol}, one has a constant $C = C(n,p,\Lambda)>0$ such that:
\begin{align} \label{eq:estbound1-2}
\left( \fint_{B_{8r}(\xi_4)}{|\nabla u_k - \nabla v_k|^{q}dx} \right)^{\frac{1}{q}} &\le C\left[\tilde{\mathcal{F}}_{8r}(\mu_k) \right]^{\frac{1}{p-1}} +  C  \tilde{\mathcal{F}}_{8r}(\mu_k) \left(\fint_{B_{8r}(\xi_4)}{|\nabla u_k|^{q}dx} \right)^{\frac{2-p}{q}},
\end{align}
where the function $\tilde{\mathcal{F}}_{8r}$ is defined by
$$\tilde{\mathcal{F}}_{8r}(\mu_k) = \left((8r)^{m}\fint_{B_{8r}(\xi_4)} |\mu_k|^{m}dx \right)^{\frac{1}{m}}.$$ 
Moreover, following the reverse H{\"o}lder's inequality in Lemma~\ref{lem:boundary-compare} with $\rho = 4r$ and notice that $B_{4r}(x)\subset B_{8r}(\xi_4)$, one has
\begin{align*}
\left(\fint_{B_{2r}(x)}|\nabla v_k|^{\Theta_0} dx\right)^{\frac{1}{\Theta_0}} &\leq C \left(\fint_{B_{4r}(x)}|\nabla v_k|^{p-1} dx\right)^{\frac{1}{p-1}}  \le C \left(\fint_{B_{8r}(\xi_4)}|\nabla v_k|^{p-1} dx\right)^{\frac{1}{p-1}},
\end{align*}
which follows from H{\"o}lder's inequality that
\begin{align}\nonumber
\left(\fint_{B_{2r}(x)}|\nabla v_k|^{\Theta_0} dx\right)^{\frac{1}{\Theta_0}} &  \le C \left(\fint_{B_{8r}(\xi_4)}|\nabla v_k|^{q} dx\right)^{\frac{1}{q}} \\ \label{eq:estbound2-2}
& \le  C \left(\fint_{B_{8r}(\xi_4)}|\nabla u_k|^{q} dx\right)^{\frac{1}{q}} + C \left(\fint_{B_{8r}(\xi_4)}|\nabla u_k - \nabla v_k|^{q} dx\right)^{\frac{1}{q}}.
\end{align}
As the proof in the first case, we first obtain the following estimate 
\begin{align}\nonumber
d_{\mathbb{U}_{2r}^x}(B_r(x);\varepsilon^{-a}\lambda) & \le \mathcal{L}^n\left(\left\{\mathbf{M}^{2r}\left(\chi_{B_{2r}(x)}|\nabla (u-u_k)|^{q}\right)^{\frac{1}{q}}>3^{-\frac{1}{q}}\varepsilon^{- a}\lambda\right\}\cap B_r(x)\right) \\ \nonumber
& \qquad +   \mathcal{L}^n\left(\left\{\mathbf{M}^{2r}\left(\chi_{B_{2r}(x)}|\nabla (u_k-v_k)|^{q}\right)^{\frac{1}{q}}>3^{-\frac{1}{q}}\varepsilon^{- a}\lambda\right\}\cap B_r(x)\right)  \\  \nonumber
& \qquad \qquad +\mathcal{L}^n\left(\left\{\mathbf{M}^{2r}\left(\chi_{B_{2r}(x)}|\nabla v_k|^{q}\right)^{\frac{1}{q}}>3^{-\frac{1}{q}}\varepsilon^{- a}\lambda\right\}\cap B_r(x)\right),
\end{align}
and then applying Lemma~\ref{lem:boundM} to yield that
\begin{align}\nonumber
d_{\mathbb{U}_{2r}^x}(B_r(x);\varepsilon^{-a}\lambda) &\le \frac{Cr^n}{(\varepsilon^{-a}\lambda)^{q}}\left[ \fint_{B_{2r}(x)}{|\nabla u - \nabla u_k|^{q}dx} + \fint_{B_{2r}(x)}{|\nabla u_k - \nabla v_k|^{q}dx} \right]  \\ \nonumber  
& \qquad + \frac{Cr^n}{(\varepsilon^{-a}\lambda)^{\Theta_0}}\fint_{B_{2r}(x)}{|\nabla v_k|^{\Theta_0}dx} \\ \nonumber
& \le \frac{Cr^n}{(\varepsilon^{-a}\lambda)^{q}}\left[ \fint_{B_{8r}(\xi_4)}{|\nabla u - \nabla u_k|^{q}dx} + \fint_{B_{8r}(\xi_4)}{|\nabla u_k - \nabla v_k|^{q}dx} \right]  \\ \label{eq:bt12-2}   
& \qquad + \frac{Cr^n}{(\varepsilon^{-a}\lambda)^{\Theta_0}}\fint_{B_{2r}(x)}{|\nabla v_k|^{\Theta_0}dx}.
\end{align}
Taking into account~\eqref{eq:estbound1-2} and~\eqref{eq:estbound2-2} to~\eqref{eq:bt12-2}, there holds  
\begin{align*}
d_{\mathbb{U}_{2r}^x}(B_r(x);\varepsilon^{-a}\lambda) &\le  \frac{Cr^n}{(\varepsilon^{-a}\lambda)^{q}} \fint_{B_{8r}(\xi_4)}{|\nabla u - \nabla u_k|^{q}dx} \\
& \quad + \frac{Cr^n}{(\varepsilon^{-a}\lambda)^{q}} \left[ \tilde{\mathcal{F}}_{8r}(\mu_k)^{\frac{1}{p-1}}+ \tilde{\mathcal{F}}_{8r}(\mu_k) \left(\fint_{B_{8r}(\xi_4)}{|\nabla u_k|^{q}dx} \right)^{\frac{2-p}{q}} \right]^{q} \\ 
& \qquad + \frac{Cr^n}{(\varepsilon^{-a}\lambda)^{\Theta_0}} \left[\left(\fint_{B_{8r}(\xi_4)}{|\nabla u_k|^{q}dx} \right)^{\frac{1}{q}} +  \tilde{\mathcal{F}}_{8r}(\mu_k)^{\frac{1}{p-1}} \right.  \\ & \left.  \qquad \quad + \tilde{\mathcal{F}}_{8r}(\mu_k) \left(\fint_{B_{8r}(\xi_4)}{|\nabla u_k|^{q}dx} \right)^{\frac{2-p}{q}} \right]^{\Theta_0}.
\end{align*}
Sending $k \to \infty$ and using Lemma~\ref{lem:conver-grad-u}, the above inequality becomes
\begin{align}\nonumber
d_{\mathbb{U}_{2r}^x}(B_r(x);\varepsilon^{-a}\lambda) &\le  \frac{Cr^n}{(\varepsilon^{-a}\lambda)^{q}} \left[ \tilde{\mathcal{F}}_{8r}(\mu)^{\frac{1}{p-1}}+ \tilde{\mathcal{F}}_{8r}(\mu) \left(\fint_{B_{8r}(\xi_4)}{|\nabla u|^{q}dx} \right)^{\frac{2-p}{q}} \right]^{q} \\ \nonumber
& \quad + \frac{Cr^n}{(\varepsilon^{-a}\lambda)^{\Theta_0}} \left[\left(\fint_{B_{8r}(\xi_4)}{|\nabla u|^{q}dx} \right)^{\frac{1}{q}} + \tilde{\mathcal{F}}_{8r}(\mu)^{\frac{1}{p-1}} \right.  \\ \label{eq:bt20-2} 
& \left.  \qquad + \tilde{\mathcal{F}}_{8r}(\mu) \left(\fint_{B_{8r}(\xi_4)}{|\nabla u|^{q}dx} \right)^{\frac{2-p}{q}} \right]^{\Theta_0}.
\end{align}
It is similar to the previous case, one estimates all terms on the right-hand side of~\eqref{eq:bt20-2} by using assumption~\eqref{eq:bt7} with the fact that 
$$B_{8r}(\xi_4) \subset B_{13r}(\xi_2) \cap B_{13r}(\xi_3). $$
With this notice, from~\eqref{eq:bt7} one obtains that
\begin{align}\label{eq:btp1-2}
\left(\fint_{B_{8r}(\xi_4)}|\nabla u|^{q}dx\right)^{\frac{1}{q}} & \leq C \left(	\fint_{B_{13r}(\xi_2)}|\nabla u|^{q}dx \right)^{\frac{1}{q}}  \le C \mathbb{U}(\xi_2) \le C \lambda,
\end{align}
and
\begin{align}\label{eq:btp2-2}
\tilde{\mathcal{F}}_{8r}(\mu) & \le  C\left({(13r)^{m}\fint_{B_{13r}(\xi_3)}|\mu|^{m}dx}\right)^{\frac{1}{m}} \le C \Pi(\xi_3) \le C (\varepsilon^{b}\lambda)^{p-1}.
\end{align}
Substituting~\eqref{eq:btp1-2} and~\eqref{eq:btp2-2} into~\eqref{eq:bt20-2}, we may conclude that
\begin{align}\notag
d_{\mathbb{U}_{2r}^x}(B_r(x);\varepsilon^{-a}\lambda) \le C r^n  \left(\varepsilon^{(a+b)q} + \varepsilon^{(a+b(p-1))q}\right)  + C r^n \varepsilon^{a\Theta_0}, 
\end{align}		
which guarantees~\eqref{est:lem-N3} by the same value of $a$, $b$ as in the first case. This ends the proof of Lemma~\ref{lem:N3}. 
\end{proof}

In order to obtain the level-set inequalities in~\eqref{eq:main-dist} and~\eqref{eq:main-dist-2}, the main idea is to use the following lemma which is well-known as a version of the Calder\'on-Zygmund-Krylov-Safonov decomposition, where its proof can be found in~\cite{CC1995}.
\begin{lemma}\label{lem:mainlem}
Let $\varepsilon \in (0,1)$, $R_1>0$ and two measurable sets $V\subset W\subset \Omega$ satisfying the following properties:
\begin{enumerate}
\item[i)] $\mathcal{L}^n(V)<\varepsilon \mathcal{L}^n(B_{R_1})$;
\item[ii)] for all $x\in \Omega$ and $r\in (0,R_1]$, if $\mathcal{L}^n(V\cap B_r(x)) \geq \varepsilon \mathcal{L}^n(B_r(x))$ then $B_r(x)\cap \Omega\subset W$.
\end{enumerate} 	
Then there exists a positive constant $C$ depending on $n$ such that $\mathcal{L}^n(V) \leq C\varepsilon \mathcal{L}^n(W)$.
\end{lemma}

\begin{proof}[Proof of Theorem~\ref{theo:lambda_estimateA}]
Let $\frac{n}{2n-1} < q < \frac{nm(p-1)}{n-m}$ and $u$ be a renormalized solution to~\eqref{eq:main}. It is easy to see that
\begin{align*}
d_{\mathbb{U}}(\varepsilon^{-a}\lambda) \le \mathcal{L}^n (\{\mathbb{U}>\varepsilon^{-a}\lambda, \, \Pi\le \varepsilon^{b}\lambda \}) + d_{\Pi}(\varepsilon^{b}\lambda).
\end{align*}
In order to prove~\eqref{eq:main-dist}, it is enough to show that there exist three constants $a \in (0,1)$, $b \in \mathbb{R}$ and $\varepsilon_0>0$ such that
\begin{align} \label{eq:mainlambda} 
\mathcal{L}^n \left(V_{\lambda,\varepsilon}\right) \le C\varepsilon \mathcal{L}^n\left(W_{\lambda}\right), \quad \forall \lambda>0, \ \varepsilon \in (0,\varepsilon_0).
\end{align}
Here $V_{\lambda,\varepsilon}$ and $W_{\lambda}$ are respectively defined by
\begin{align*}
V_{\lambda,\varepsilon} = \left\{\mathbb{U}>\varepsilon^{-a}\lambda, \, \Pi\le \varepsilon^{b}\lambda \right\} \ \mbox{ and } \ W_{\lambda} = \left\{\mathbb{U} > \lambda\right\}.
\end{align*} 
We may assume that $V_{\lambda,\varepsilon} \neq \emptyset$, which leads to the existence of $\xi_0 \in \Omega$ such that $\Pi(\xi_0) \le \varepsilon^{b}\lambda$. Thanks to Lemma~\ref{lem:N1}, one has
\begin{equation}\label{assump:i}
\mathcal{L}^n(V_{\lambda,\varepsilon}) \le d_{\mathbb{U}}(\varepsilon^{-a}\lambda) \leq \varepsilon \mathcal{L}^n ({B_{R_0}}(0)),  
\end{equation}
for all $\lambda>0$, where $R_0=\min\{D_0,r_0\}$. Moreover one may verify that for all $x\in \Omega$, $r\in (0,R_0]$ and $\lambda>0$, the following statement does hold:
\begin{equation}\label{assump:ii}
\mathcal{L}^n(V_{\lambda,\varepsilon}\cap B_r(x)) \geq C\varepsilon \mathcal{L}^n(B_r(x)) \Longrightarrow B_r(x)\cap \Omega\subset W_\lambda.
\end{equation}
Indeed, let us suppose that $B_r(x)\cap \Omega \cap W^c_\lambda\not= \emptyset$ and $V_{\lambda,\varepsilon}\cap B_r(x)\not = \emptyset$. Then, there exist $\xi_2$, $\xi_3\in B_r(x)\cap \Omega$ such that $\mathbb{U}(\xi_2) \le \lambda$ and $\Pi(\xi_3) \le \varepsilon^{b}\lambda$. Thanks to Lemma~\ref{lem:N2} and Lemma~\ref{lem:N3}, one can find suitable parameters $a \in (0,1)$, $b\in \mathbb{R}$ and $\varepsilon_0>0$ such that
\begin{align*}
\mathcal{L}^n(V_{\lambda,\varepsilon}\cap B_r(x)) \le d_{\mathbb{U}}(B_r(x);\varepsilon^{-a}\lambda) \leq d_{\mathbb{U}_{2r}^x}(B_r(x);\varepsilon^{-a}\lambda) \le C \varepsilon \mathcal{L}^n(B_r(x)),
\end{align*}
which implies~\eqref{assump:ii} by contradiction. From~\eqref{assump:i} and~\eqref{assump:ii}, the inequality~\eqref{eq:mainlambda} holds by applying Lemma~\ref{lem:app}. The proof is complete.
\end{proof}

\begin{proof}[Proof of Theorem \ref{theo:good-lambda-2B}] The proof of this theorem is slightly the same as the proof of Theorem~\ref{theo:lambda_estimateA}. The difference lies on the using of Lemma~\ref{lem:N1} which will be replaced by the next estimate. For any $\lambda>\varepsilon^{-b} \rho^{-\frac{n}{q}} \|\nabla u\|_{L^{q}(B_{10\rho}(x)\cap \Omega)}$, thanks to Lemma~\ref{lem:boundM} we have the following estimate
\begin{align*}
d_{\tilde{\mathbb{U}}}(B_{\rho}(x);\varepsilon^{-a}\lambda) & \leq \frac{C}{(\varepsilon^{-a}\lambda)^q} \int_{\Omega} \chi_{B_{10\rho}(x)}|\nabla u|^{q} dx \\
&  \leq \frac{C}{(\varepsilon^{-a}\varepsilon^{-b} \rho^{-\frac{n}{q}} \|\nabla u\|_{L^{q}(B_{10\rho}(x)\cap \Omega)})^q} \int_{\Omega} \chi_{B_{10\rho}(x)}|\nabla u|^{q} dx \\
&  \le C \varepsilon^{(a+b)q} \mathcal{L}^{n}(B_{10\rho}(x)) \\
&  \le  \varepsilon \mathcal{L}^{n}(B_{10\rho}(x)).
\end{align*}
As in the proof of Theorem~\ref{theo:lambda_estimateA}, we recall that $a$, $b$ will be chosen such that $(a+b)q > 1$ which guarantees the last inequality for $\varepsilon$ small enough. The other steps of the proof will be performed by the same way as in Theorem~\ref{theo:lambda_estimateA}. 
\end{proof}

\subsection{Gradient estimate in Lorentz spaces}
\begin{proof}[Proof of Theorem \ref{theo:mainC}]	
In what follows we prove Theorem~\ref{theo:mainC} only for the case $t \neq \infty$, and for $t=\infty$ the proof is similar. Let us fix $\frac{n}{2n-1}<q<\frac{nm(p-1)}{n-m}$.  Thanks to Theorem~\ref{theo:lambda_estimateA},  there exist constants $\Theta_0>p$, $a =\Theta_0^{-1}$, $b \in \mathbb{R}$, $C>0$ and $0<\varepsilon_0<1$  such that the following inequality 
\begin{align}\label{eq:EF}
d_{\mathbb{U}}(\varepsilon^{-a}\lambda) \le C \varepsilon d_{\mathbb{U}}(\lambda) + d_{\Pi}(\varepsilon^{b}\lambda)
\end{align}
holds for any $\varepsilon \in (0,\varepsilon_0)$ and $\lambda>0$, where
$\mathbb{U} = (\mathbf{M}(|\nabla u|^{q}))^{\frac{1}{q}}$, and $\Pi = (\mathbf{M}_{m}(|\mu|^{m}))^{\frac{1}{m(p-1)}}$. By changing of variables from $\lambda$ to $\varepsilon^{-a}\lambda$ in the standard definition of Lorentz space, one has
\begin{align}\label{eq:m-1}
\|\mathbb{U}\|_{L^{s,t}(\Omega)}^t & = s \int_0^\infty{\lambda^{t-1} \left[d_{\mathbb{U}}(\lambda)\right]^{\frac{t}{s}} {d\lambda}} = {\varepsilon^{-at}} s \int_0^\infty{ \lambda^{t-1} \left[d_{\mathbb{U}}(\varepsilon^{-a}\lambda)\right]^{\frac{t}{s}} {d\lambda}}.
\end{align} 
Substituting~\eqref{eq:EF} into~\eqref{eq:m-1}, we obtain that
\begin{align}\label{eq:2} 
\|\mathbb{U}\|_{L^{s,t}(\Omega)}^{t-1} \le C {\varepsilon^{-at + \frac{t}{s}}} s \int_0^\infty{ \lambda^{t-1} \left[d_{\mathbb{U}}(\lambda)\right]^{\frac{t}{s}} {d\lambda}} + C \varepsilon^{-at} s \int_0^\infty{ \lambda^{t-1} \left[d_{\Pi}(\varepsilon^b\lambda)\right]^{\frac{t}{s}} {d\lambda}}.
\end{align}
Changing of variables for the last integral in~\eqref{eq:2}, one gets
\begin{align*}
\|\mathbb{U}\|_{L^{s,t}(\Omega)}^t  \le C {\varepsilon^{\left(-a + \frac{1}{s}\right)t}}  \|\mathbb{U}\|_{L^{s,t}(\Omega)}^t  +  C \varepsilon^{-at-bt} \|\Pi\|_{L^{s,t}(\Omega)}^t,
\end{align*}
which deduces to
\begin{align}\label{eq:3}
\|\mathbb{U}\|_{L^{s,t}(\Omega)}  \le C {\varepsilon^{-a + \frac{1}{s}}} \|\mathbb{U}\|_{L^{s,t}(\Omega)} +  C \varepsilon^{-a-b} \|\Pi\|_{L^{s,t}(\Omega)}.
\end{align}
For any $0<s<a^{-1}= \Theta_0$ and $0<t<\infty$, in~\eqref{eq:3} we may choose $\varepsilon\in (0,\varepsilon_0)$ small enough such that $C {\varepsilon^{-a + \frac{1}{s}}} \le 1/2$ and in conclusion we have obtained~\eqref{eq:theo-mainC}.
\end{proof}

\subsection{Gradient estimate in Lorentz-Morrey spaces}
In this subsection, we prove the Lorentz-Morrey gradient estimate for renormalized solution to~\eqref{eq:main}. The following standard lemma is useful for our proof.
\begin{lemma}\label{lem:app}
Let $f \in L^{s,t;\kappa}(\Omega)$ for $0<s<\infty$, $0<t\le \infty$ and $0<\kappa \le n$. For $0< \sigma \le \frac{\kappa}{s}$, there exists a constant $C=C(n,s,\kappa,\sigma)>0$ such that
\begin{align}\label{est:lem-app-2}
\mathbf{M}_{\sigma}(f)(y) \le C \left(\|f\|_{L^{s,t; \kappa}(\Omega)}\right)^{\frac{\sigma s}{\kappa}} \left[\mathbf{M}(f)(y)\right]^{1-\frac{\sigma s}{\kappa}} ,
\end{align}
for $\rho>0$ and $y \in \Omega$. In particular, there holds
\begin{align}\label{est:lem-app-1}
\|\mathbf{M}_{\frac{\kappa}{s}}^{D_0}(f)\|_{L^{\infty}(\Omega)} \le C \|f\|_{L^{s,t;\kappa}(\Omega)}.
\end{align}
\end{lemma}
\begin{proof}
Let $\rho>0$ and $y \in \Omega$. For any $0< \alpha\le 1$, we have
\begin{align*}
\rho^{\sigma-n} \int_{B_{\rho}(y)} f(x)dx & =  \left(\rho^{\frac{\sigma}{\alpha}-n} \int_{B_{\rho}(y)} f(x)dx\right)^{\alpha} \left(\rho^{-n} \int_{B_{\rho}(y)} f(x)dx\right)^{1-\alpha} \\
& \le C \left(\rho^{\frac{\sigma}{\alpha}-n} \rho^{n-\frac{n}{s}} \|f\|_{L^{s,\infty}(B_{\rho}(y))}\right)^{\alpha} \left[\mathbf{M}(f)(y)\right]^{1-\alpha} \\
& \le C \left(\rho^{\frac{\sigma}{\alpha}-\frac{n}{s}}  \|f\|_{L^{s,t}(B_{\rho}(y))}\right)^{\alpha} \left[\mathbf{M}(f)(y)\right]^{1-\alpha}\\
& \le C \left(\|f\|_{L^{s,t;\frac{\sigma s}{\alpha}}(\Omega)}\right)^{\alpha} \left[\mathbf{M}(f)(y)\right]^{1-\alpha}.
\end{align*}
Let us take $\alpha = \frac{\sigma s}{\kappa}$, we obtain~\eqref{est:lem-app-2}. By choosing $\alpha=1$ and taking the supremum both sides of this inequality for all $0<\rho<D_0$ and $y \in \Omega$, we obtain~\eqref{est:lem-app-1} which completes the proof.
\end{proof}

\begin{proof}[Proof of Theorem~\ref{theo:mainD} ]	
For simplicity of notation, we denote  $B_{\rho} := B_\rho(x)$ and $B_{10\rho}: = B_{10\rho}(x)$ with $0<\rho<D_0$ and $x \in \Omega$. For every $q \in \left(\frac{n}{2n-1},\frac{nm(p-1)}{n-m}\right)$, let us define
$$\tilde{\mathbb{U}} = (\mathbf{M}(\chi_{B_{10\rho}}|\nabla u|^{q}))^{\frac{1}{q}}, \quad \tilde{\Pi} = (\mathbf{M}_{m}(\chi_{B_{10\rho}}|\mu|^{m}))^{\frac{1}{m(p-1)}}.$$
Thanks to Theorem~\ref{theo:good-lambda-2B}, one can find constants $a \in (0,1)$, $b \in \mathbb{R}$, $\varepsilon_0>0$ and  $C>0$ such that the following estimate holds 	
\begin{align} \label{eq:esp}
 d_{\tilde{\mathbb{U}}}(B_{\rho};\varepsilon^{-a}\lambda) \le C \varepsilon d_{\tilde{\mathbb{U}}}(B_{\rho};\lambda) + d_{\tilde{\Pi}}(B_{\rho};\varepsilon^{b}\lambda),
\end{align}
for any $\varepsilon\in (0,\varepsilon_0)$ and $\lambda>\lambda_0$, where 
\begin{align}\label{eq:lambda-0}
\lambda_0 = \varepsilon^{-b} \rho^{-\frac{n}{q}} \|\nabla u\|_{L^{q}(B_{10\rho}\cap \Omega)}.
\end{align}
For the convenience of the reader, let us denote
 \begin{align}\label{eq:alpha}
 \alpha := \frac{ m s \kappa (p-1)}{\kappa-ms} \ \mbox{ and } \ \beta := \frac{ m t \kappa (p-1)}{\kappa-mt}.
 \end{align}
By changing of variables in the definition of Lorentz norm we obtain that
 \begin{align}\label{eq:tp1}
 \|\tilde{\mathbb{U}}\|_{L^{\alpha,\beta}(B_{\rho})}^\beta & = \varepsilon^{-\beta a} \alpha\int_{0}^{\infty}\lambda^{\beta-1} \left[d_{\tilde{\mathbb{U}}}(B_{\rho};\varepsilon^{-a}\lambda)\right]^{\frac{\beta}{\alpha}} d\lambda \notag \\
 & = \varepsilon^{-\beta a} \alpha\int_{0}^{\lambda_0}\lambda^{\beta-1} \left[d_{\tilde{\mathbb{U}}}(B_{\rho};\varepsilon^{-a}\lambda)\right]^{\frac{\beta}{\alpha}} d\lambda \notag \\ 
 & \qquad \qquad + \varepsilon^{-\beta a} \alpha\int_{\lambda_0}^{\infty}\lambda^{\beta-1}\left[d_{\tilde{\mathbb{U}}}(B_{\rho};\varepsilon^{-a}\lambda)\right]^{\frac{\beta}{\alpha}} d\lambda.
 \end{align}
We remark that~\eqref{eq:esp} holds for any $\lambda>\lambda_0$, it follows from~\eqref{eq:tp1} that
\begin{align*}
 \|\tilde{\mathbb{U}}\|_{L^{\alpha,\beta}(B_{\rho})}^{\beta} & \leq C\varepsilon^{-\beta a}\lambda_0^{\beta}\mathcal{L}^n\left( B_{\rho}\right)^{\frac{\beta}{\alpha}} + C\varepsilon^{-\beta a+\frac{\beta}{\alpha}}\alpha\int_{\lambda_0}^{\infty}\lambda^{\beta-1} \left[d_{\tilde{\mathbb{U}}}(B_{\rho};\lambda)\right]^{\frac{\beta}{\alpha}} d\lambda \\
 & \qquad + C\varepsilon^{-\beta a} \alpha\int_{\lambda_0}^{\infty}\lambda^{\beta-1}\left[d_{\tilde{\Pi}}(B_{\rho};\varepsilon^{b}\lambda)\right]^{\frac{\beta}{\alpha}} d\lambda\\
 & \le C\varepsilon^{-\beta a}\lambda_0^{\beta}\rho^{\frac{n\beta}{\alpha}}+ C\varepsilon^{-\beta a+\frac{\beta}{\alpha}} \|\tilde{\mathbb{U}}\|_{L^{\alpha,\beta}(B_{\rho})}^{\beta}  + C\varepsilon^{-\beta a-\beta b}\|\tilde{\Pi}\|_{L^{\alpha,\beta}(B_{\rho})}^{\beta}.
 \end{align*}
Then, it gives us the estimate:
\begin{align}\label{eq:tp3}
\|\tilde{\mathbb{U}}\|_{L^{\alpha,\beta}(B_{\rho})} & \leq  C\varepsilon^{-a}\lambda_0\rho^{\frac{n}{\alpha}}  + C\varepsilon^{-a+\frac{1}{\alpha}} \|\tilde{\mathbb{U}}\|_{L^{\alpha,\beta}(B_{\rho})}  + C\varepsilon^{-a-b}\|\tilde{\Pi}\|_{L^{\alpha,\beta}(B_{\rho})}.
\end{align}
The inequality~\eqref{eq:tp3} holds for any $\alpha >0$ and $0 < \beta< \infty$. We remark that this inequality even holds for $\beta=\infty$ by the same method. For $0< \alpha< \Theta_0:= a^{-1}$, we may choose $\varepsilon \in (0,\varepsilon_0)$ such that $C\varepsilon^{-a+\frac{1}{\alpha}} <1/2$ in~\eqref{eq:tp3}, then one obtains from~\eqref{eq:lambda-0} that
\begin{align}\label{est:M-10-1} 
\|\tilde{\mathbb{U}}\|_{L^{\alpha,\beta}(B_{\rho})} & \leq  C \rho^{\frac{n}{\alpha}-\frac{n}{q}} \|\nabla u\|_{L^{q}(B_{10\rho}\cap \Omega)} + C\|\tilde{\Pi}\|_{L^{\alpha,\beta}(B_{\rho})}.
 \end{align} 
Applying Lemma~\ref{lem:res7.2H} with $\sigma = \frac{\kappa}{s}$ and $\delta = \frac{\kappa-ms}{ms(p-1)}$ satisfying 
$$m+ m(p-1)(1-\beta_0) < \frac{\kappa}{s} \le n,$$ 
it gives us the following estimate
 \begin{align*}
\|\nabla u\|_{L^{q}(B_{10\rho}\cap \Omega)}\le C \rho^{\frac{n}{q}-\delta} \|\mathbf{M}_{\frac{\kappa}{s}}^{D_0}(|\mu|^m)\|_{L^\infty(\Omega)}^{\frac{1}{m(p-1)}},
 \end{align*}
which deduces from~\eqref{est:M-10-1} that
\begin{align}\label{est:LM-1}
 \rho^{\delta-\frac{n}{\alpha}}\|\tilde{\mathbb{U}}\|_{L^{\alpha,\beta}(B_{\rho}(x))} & \leq  C \|\mathbf{M}_{\frac{\kappa}{s}}^{D_0}(|\mu|^m)\|_{L^\infty(\Omega)}^{\frac{1}{m(p-1)}} + C \rho^{\delta-\frac{n}{\alpha}} \|\tilde{\Pi}\|_{L^{\alpha,\beta}(B_{\rho}(x))}.
 \end{align}
Here it is very easy to check that $\delta \alpha = \kappa$ using the definition of $\alpha$ in~\eqref{eq:alpha}. By taking the supremum both sides of~\eqref{est:LM-1} for $0<\rho<D_0$ and $x \in \Omega$, it guarantees that
 \begin{align}\label{est:I}
 \|\nabla u\|_{L^{\alpha,\beta; \kappa}(\Omega)} \leq C (I_1 + I_2),
 \end{align}
 where $I_1$ and $I_2$ are defined by
 \begin{align}\label{est:I2-1}
 I_1 := \|\mathbf{M}_{\frac{\kappa}{s}}^{D_0}(|\mu|^m)\|_{L^\infty(\Omega)}^{\frac{1}{m(p-1)}} \ \mbox{ and } \ 
 I_2 :=  \sup_{0<\rho<D_0,\, x \in \Omega} \rho^{\frac{\kappa - n}{\alpha}} \|\tilde{\Pi}\|_{L^{\alpha,\beta}(B_{\rho}(x))}.
 \end{align}
 Applying~\eqref{est:lem-app-2} in Lemma~\eqref{lem:app}, one easily estimates $I_1$ as
 \begin{align}\label{est:I1}
 I_1 \le C \||\mu|^m\|_{L^{s,t;\kappa}(\Omega)}^{\frac{1}{m(p-1)}} .
 \end{align}
 It is necessary to estimate $I_2$ by the same norm  in $L^{s,t;\kappa}(\Omega)$. For any $y \in B_{\rho}(x)$, thanks to~\eqref{est:lem-app-1} in Lemma~\ref{lem:app} we have
 \begin{align*}
 (\mathbf{M}_m(\chi_{B_{10\rho}(x)}|\mu|^m))(y) \le C \left[(\mathbf{M}(\chi_{B_{10\rho}(x)}|\mu|^m))(y)\right]^{1-\frac{m s}{\kappa}} \left(\||\mu|^m\|_{L^{s,t; \kappa}(B_{10\rho}(x))}\right)^{\frac{m s}{\kappa}},
 \end{align*}
which implies that
 \begin{align*}
 \|\tilde{\Pi}\|_{L^{\alpha,\beta}(B_{\rho}(x))}  & = \left[\|\mathbf{M}_m(\chi_{B_{10\rho}(x)}|\mu|^m))\|_{L^{\frac{\alpha}{m(p-1)},\frac{\beta}{m(p-1)}}(B_{\rho}(x))}\right]^{\frac{1}{m(p-1)}} \\
 & \le C \left\|(\mathbf{M}(\chi_{B_{10\rho}(x)}|\mu|^m))^{\left(1-\frac{m s}{\kappa}\right)}\right\|_{L^{\frac{\alpha}{m(p-1)},\frac{\beta}{m(p-1)}}(B_{\rho}(x))}^{\frac{1}{m(p-1)}} \||\mu|^m\|_{L^{s,t;\kappa}(B_{10\rho}(x))}^{\frac{s}{(p-1)\kappa}} \\
 & \le C \left\|(\mathbf{M}(\chi_{B_{10\rho}(x)}|\mu|^m))\right\|_{L^{\frac{\alpha(\sigma-m)}{m(p-1)\sigma},\frac{\beta(\sigma-m)}{m(p-1)\sigma}}(B_{\rho}(x))}^{\frac{\sigma - m}{m(p-1)\sigma}} \||\mu|^m\|_{L^{s,t;\kappa}(B_{10\rho}(x))}^{\frac{s}{(p-1)\kappa}}\\
 & = C \left\|(\mathbf{M}(\chi_{B_{10\rho}(x)}|\mu|^m))\right\|_{L^{s,t}(B_{\rho}(x))}^{\frac{\kappa - ms}{m(p-1)\kappa}} \||\mu|^m\|_{L^{s,t;\kappa}(B_{10\rho}(x))}^{\frac{s}{(p-1)\kappa}}.
 \end{align*}
 Using the boundedness of the Hardy-Littlewood maximal function $\mathbf{M}$, ones obtains that
 \begin{align*}
\|\tilde{\Pi}\|_{L^{\alpha,\beta}(B_{\rho}(x))}  \le C \||\mu|^m\|_{L^{s,t}(B_{10\rho}(x))}^{\frac{\kappa - ms}{m(p-1)\kappa}} \||\mu|^m\|_{L^{s,t;\kappa}(B_{10\rho}(x))}^{\frac{s}{(p-1)\kappa}}.
 \end{align*}
By the definition of Lorentz-Morrey norm with remark that $\frac{s}{\alpha} = \frac{\kappa - ms}{m(p-1)\kappa}$, we deduce from the above inequality that
 \begin{align}\label{est:I2-2}
 \|\tilde{\Pi}\|_{L^{\alpha,\beta}(B_{\rho}(x))}  \le    C \rho^{\frac{n - \kappa}{\alpha}} \||\mu|^m\|_{L^{s,t;\kappa}(\Omega)}^{\frac{1}{m(p-1)}}.
 \end{align}
 Combining~\eqref{est:I2-1} and~\eqref{est:I2-2}, we get that
 \begin{align}\label{est:I2}
 I_2 \le C \||\mu|^m\|_{L^{s,t;\kappa}(\Omega)}^{\frac{1}{m(p-1)}}.
 \end{align}
Taking into account~\eqref{est:I1} and~\eqref{est:I2} to~\eqref{est:I}, we may conclude~\eqref{est:LM}. Finally, we note that all hypotheses that we need on parameters $s$, $\kappa$ are
\begin{align*}
m+m(p-1)(1-\beta_0)< \frac{\kappa}{s} \le n, \mbox{ and } 0<\frac{m(p-1)s\kappa}{\kappa-ms}< \Theta_0,
\end{align*}
which are equivalent to assumption~\eqref{est:cond-s}. Moreover, we remark that
\begin{align*}
\frac{\kappa}{n} < \frac{\kappa \Theta_0}{m\Theta_0 + m(p-1)\kappa}  \ \mbox{ if and only if } \ \kappa < \frac{(n-m)\Theta_0}{(p-1)m},
\end{align*}
and
\begin{align*}
\frac{\kappa}{n} < \frac{\kappa}{m+m(1-\beta_0)(p-1)}  \ \mbox{ if and only if } \ m < \frac{n}{1 + (1-\beta_0)(p-1)},
\end{align*}
which is always true for $m < m^{**}$. The proof of~\eqref{est:LM} is complete.
\end{proof}

\end{document}